\documentclass[12pt,reqno,a4paper]{article}
\usepackage{euscript}
\usepackage{amsmath,amssymb,amsfonts,amsthm}
\usepackage{mathtools}
\usepackage{dsfont}
\usepackage{sectsty}
\usepackage[colorlinks=true, linkcolor=blue, citecolor=magenta]{hyperref}
\newtheoremstyle{special}%
{}%
{}%
{}%
{}%
{\scshape}%
{.}%
{.5em}%
{}

\newtheorem{theorem}{Theorem}
\newtheorem{proposition}[theorem]{Proposition}
\newtheorem{lemma}[theorem]{Lemma}
\newtheorem{cor}[theorem]{Corollary}
\newtheorem{definition}[theorem]{Definition}

\theoremstyle{special}
\newtheorem{remark}[theorem]{Remark}
\newtheorem{example}[theorem]{Example}

\allsectionsfont{\centering\mdseries\normalsize\scshape}

\renewcommand{\epsilon}{\varepsilon}
\renewcommand{\L}{\mathcal{L}}

\def\Id{\text{\rm Id}}

\def\N{\mathbb{N}}
\def\Z{\mathbb{Z}}
\def\R{\mathbb{R}}
\def\B{\mathcal{B}}
\def\C{\mathcal C}

\DeclareMathOperator{\var}{var}
\DeclareMathOperator{\esssup}{esssup}
\DeclareMathOperator{\essinf}{essinf}

\DeclareMathOperator{\Cov}{Cov}

\title{A vector-valued almost sure invariance principle for random expanding on average cocycles}

\date{\today}

\author{D. Dragi\v cevi\' c \footnote{Department of Mathematics, University of Rijeka, Rijeka Croatia. {\tt E-mail: ddragicevic@math.uniri.hr}.}\and 
	Y. Hafouta\footnote{Department of Mathematics, The Ohio State University. {\tt E-mail: yeor.hafouta@mail.huji.ac.il, hafuta.1@osu.edu}.}\and
	J. Sedro\footnote{Laboratoire  de Probabilit\'es,  Statistique et Mod\'elisation  (L.P.S.M.), Sorbonne  Universit\'e. {\tt E-mail: sedro@lpsm.paris}.}}

\begin{document}
	\maketitle
	\begin{abstract}
		We obtain a quenched  vector-valued almost sure invariance principle (ASIP) for  random expanding on average cocycles. This is achieved by combining the adapted version  of Gou\"ezel's approach for establishing ASIP (developed in~\cite{DH3}) and the recent construction of the so-called adapted norms (carried out in~\cite{DDS}), which in some sense eliminate the non-uniformity of the decay of correlations. 
		For real-valued observables,  we also show that the martingale approximation technique is applicable in our setup, and that it yields the ASIP with  better error rates. 
		Finally, we present an example showing the necessity of the scaling condition \eqref{eq:condobservable}, answering a question of \cite{DDS}.
	\end{abstract} 
	

	\section{Introduction}\label{Sec1}
	The almost sure invariance principle (ASIP)  is a  powerful statistical tool that,
	given a sequence of vector-valued random variables $A_0,A_1,A_1,\ldots$, provides a coupling with an independent sequence of Gaussian random vectors $Z_0,Z_1,Z_2,\ldots$ such that
	\begin{equation}\label{sn}
	\left|\sum_{j=0}^{n-1}A_j-\sum_{j=0}^{n-1}Z_j\right|=o(s_n),
	\end{equation}
	where $s_n$ is called the \emph{rate} of the ASIP, 
	and the $L^2$-norm of the sum $\sum_{j=0}^{n-1}Z_j$ has the form $s_n(1+o(1))$. Among several other limit theorems, ASIP implies the central and the functional central limit theorem (see~\cite{PS} for details).
	
	\paragraph{ASIP for deterministic dynamics.}
	In the context of \emph{deterministic} dynamics, one starts with a transformation $T$ acting on a space $X$ that admits a (physical) invariant  measure $\mu$. For sufficiently regular observables $\psi \colon X \to \R^d$, we consider the process $\psi, \psi \circ T, \psi \circ T^2, \ldots $ on the probability space $(X, \mu)$, and we are interested in formulating sufficient conditions under which it satisfies the ASIP. It is expected that this will occur when $T$ is sufficiently chaotic, i.e. when $T$ exhibits some form of hyperbolicity. 
	
	We emphasize that this problem has been thoroughly studied and that the literature is vast. Among many important contributions, we mention the works of Pollicott and Sharp~\cite{MPS}, 
	Field, Melbourne and T\"or\"ok~\cite{FieldMelbourneTorok} as well as Melbourne and Nicol~\cite{MN1, MN2} (completed recently  by Korepanov~\cite{KorepanovEq}), in which the authors obtained ASIP for large classes of (nonuniformly) hyperbolic maps. In addition, we mention the recent important papers by Cuny and Merlevede~\cite{CM},  Korepanov, Kosloff and Melbourne~\cite{KZM}, Korepanov~\cite{Korepanov}, as well as Cuny, Dedecker,  Korepanov and  Merlevede~\cite{CDKM, CDKM2} in which the authors
	further improved the error rates in ASIP for a wide class of (nonuniformly) hyperbolic deterministic dynamical systems. 
	
	Finally, and most relevant to the content of the present paper,  we mention the seminal paper by Gou\"ezel~\cite{GO}, in which is developed the so-called \emph{spectral approach} for establishing the ASIP, which is applicable whenever the transfer operator associated to $T$ admits a spectral gap (on an appropriate Banach space).
	
	\paragraph{ASIP for random dynamics.}
	In the case of \emph{random dynamics}, one starts with a base space which consists of a probability space $(\Omega, \mathcal F, \mathbb P)$ together with an invertible measure-preserving and ergodic transformation $\sigma \colon \Omega \to \Omega$. Moreover, one considers a measurable family of transformations
	$(T_\omega)_{\omega \in \Omega}$ acting on some space $X$. We study random compositions of the form
	\[
	T_\omega^n=T_{\sigma^{n-1} \omega} \circ \ldots \circ T_\omega, \quad \omega \in \Omega, \ n\in \mathbb N.
	\]
	From this data, we can define the skew-product transformation $\tau \colon \Omega \times X \to \Omega \times X$ by 
	\begin{equation}\label{SP}
		\tau(\omega, x)=(\sigma \omega, T_\omega(x)), \quad (\omega, x)\in \Omega \times X.
	\end{equation}
	For any $\tau$-invariant measure $\mu$, there exists  a family $(\mu_\omega)_{\omega \in \Omega}$ of measures on $X$ such that 
	\begin{equation}\label{mu}
		\mu (A\times B)=\int_A \mu_\omega(B)\, d\mathbb P(\omega), \quad \text{for $A\in \mathcal F$ and $B\subset X$ measurable.}
	\end{equation}
	Given a sufficiently regular (random) observable $\psi \colon \Omega \times X \to \R^d$, one can either study the process 
	\begin{equation}\label{an}
		\psi, \psi \circ \tau, \ldots, \psi \circ \tau^n  \quad \text{on $(\Omega \times X, \mu)$,}
	\end{equation}
	or for $\mathbb P$-a.e. $\omega \in  \Omega$, the process
	\begin{equation}\label{q}
		\psi (\omega, \cdot), \psi (\sigma \omega, \cdot )\circ T_\omega, \ldots, \psi (\sigma^n \omega, \cdot) \circ T_\omega^n \quad \text{on $(X, \mu_\omega)$.}
	\end{equation}
	Then, limit laws related to~\eqref{an} are called \emph{annealed}, whereas those concerned with~\eqref{q} are called \emph{quenched}. We would like to stress that annealed limit theorems concern the stationary process $(\psi\circ T^n)_{n=0}^\infty$, while in the quenched case the process $(\psi (\sigma^n \omega, \cdot) \circ T_\omega^n)_{n=0}^\infty$ is not stationary, which to some extent makes the quenched limit theorems harder to prove. On the other hand, without some kind of mixing assumptions on the base map $\sigma$ annealed limit theorems cannot hold (in general). For i.i.d maps $T_{\sigma^j\omega}$  (as discussed in the next paragraph) limit theorems can be obtained by integration of the iterates of the random transfer operators (see \cite{ANV}), while for some other classes of base maps such as Markov shifts or non-uniform Young towers a different type of integration argument yields several types of limit theorems (see \cite{HaETDS}). However, it is still unclear which type of mixing conditions are necessary for annealed limit theorems to hold.

	To the best of our knowledge, the  quenched ASIP in the context of random dynamical systems was first discussed by Kifer~\cite{kifer}, where it was briefly mentioned that the techniques developed there can be used to obtain the scalar-valued quenched ASIP for random expanding dynamics. More recently, both annealed and quenched  ASIP were 
	discussed for several classes of random dynamical systems~\cite{ANV,STE,STSU}. The main idea in those papers is to show that the transfer operator associated to the skew-product transformation (see~\eqref{SP}) has a spectral gap on an appropriate space. Afterwards, one can simply apply Gou\"ezel's results~\cite{GO}. However, for this approach to work, one needs to impose strong (mixing) assumptions on the base space $(\Omega, \mathcal F, \mathbb P)$. In fact, in all of those works, $(\Omega, \mathcal F, \mathbb P)$ is a Bernoulli shift. By using martingale methods and relying on the techniques developed in~\cite{CM, HNTV}, in~\cite{DFGTV1} a quenched scalar-valued ASIP for certain classes of random piecewise expanding dynamics was obtained, without any mixing assumptions on the base space. In addition, we mention two recent papers \cite{Su1, Su2} by Su devoted to the ASIP for certain classes of random expanding systems and maps which admit a random tower extension.
	
	In order to apply directly the approach developed by Gou\"ezel~\cite{GO} for establishing the ASIP in the context of random dynamics, it seems necessary to impose mixing assumptions on the base space $(\Omega, \mathcal F, \mathbb P)$. To overcome this issue, the first two authors~\cite{DH3} have proposed  a certain adaptation of Gou\"ezel's method, by requiring a weaker control over the behavior of  covariance matrices. This adaptation enabled them to extend the ASIP result from~\cite{DFGTV1} to the case of vector-valued observables, still without any mixing assumptions on the base space. Moreover, it allowed to establish the first version of the vector-valued ASIP for broad classes of random (uniformly) hyperbolic dynamics as studied in~\cite{DH, DZ, DFGTV}.
	
	\paragraph{Contributions of the present paper.}
	Despite representing a significant progress, the abstract version of the ASIP for random dynamical systems formulated in~\cite[Theorem 4.18]{DH3} is not entirely satisfactory as it is not directly applicable to random dynamical systems which exhibit nonuniform decay of correlations, such as expanding on average systems studied by Buzzi~\cite{Buzzi}. Indeed, condition~\cite[(4.6)]{DH3} requires uniform decay of correlations: we refer to Remark~\ref{ej} for more details.
	We note that the presence of the  nonuniform decay of correlations essentially means that  $D$ in~\eqref{est1x} is allowed to depend on the random parameter $\omega$. This relaxation (with respect to previous works) is natural from the ergodic theory point of view. We refer to Remark~\ref{newrmk} for details.

	The purpose of the present paper is to fill the gap between the uniformly expanding case and the nonuniform one. More precisely, we combine techniques from~\cite{DH3} together with those developed by the first and the third author in~\cite{DDS}, to obtain the quenched vector-valued ASIP for random expanding on average. We stress that in~\cite{DDS} several other limit theorems for random expanding on average cocycles have been discussed but not the ASIP.
	 As mentioned above, in comparison with several existing result in literature (e.g. \cite{HNTV, DFGTV1, DH3}), we consider random dynamical systems which only expand on average. In another direction, 
	in comparison with the expanding on average random maps considered in \cite{ANV}, our results do not require the maps $T_{\sigma^j\omega}$ to be independent, and the observable $\psi$ considered in this paper might depend on $\omega$. In fact, in Appendix A we show that, in general, when $\psi$ does not depend on $\omega$ then the usual asymptotic statistical behaviour might fail.
	
	As in~\cite{DDS}, the main idea is to construct suitable ``adapted norms'', which enable us to absorb the nonuniformity in the decay of correlations. Unlike~\cite{DDS}, this construction is not carried for an original cocycle of transfer operators but rather for an associated cocycle of normalized transfer operators. We highlight that, 
	to the best of our understanding, one cannot simply rely on the construction developed in~\cite[Section 3.1]{DDS}. In fact, we have to restrict to a slightly smaller class of cocycles from those considered in~\cite{DDS} (see Remark~\ref{Rm}). Afterwards, it remains to verify that Theorem~\cite[Theorem 2.1]{DH3} can be applied. 
	
	We also note that this approach is completely different from the techniques in~\cite{kifer, DH1} which rely on inducing with respect to a region of the base space on which the random dynamics exhibits a uniform decay of correlations, and refer to~\cite[Section 1]{DDS} for a detailed discussion. 
	
	In the second part of the paper, by using the martingale method together with techniques developed in~\cite{CM} and~\cite{KZM}, we establish the scalar-valued ASIP for a  smaller class of observables but with better error rates (that is, with better estimates on the right hand side of~\eqref{sn}). We stress that even if we restrict to the setup from~\cite{DFGTV1, DFGTV2} (namely assume uniform decay of correlations), we obtain slightly better rates than those given in~\cite[Theorem 1]{DFGTV1}.
	\par Finally, in an Appendix A, we present an example, essentially taken from \cite[Appendix A]{Buzzi}, of a random system and an observable satisfying all of our assumptions, except for the scaling condition \eqref{eq:condobservable}, for which the asymptotic variance fails to exist: this shows the sharpness of this scaling condition, answering in particular a question posed in (the original version of) \cite{DDS}.

\paragraph{Comments and directions of future research.}
We emphasize that our results (see for example Theorem~\ref{TM}) require  a certain control over the size of an observable (condition~\eqref{eq:condobservable}). This condition does not imply that the class of our observables is small:  indeed, one can note that the observables satisfying~\eqref{eq:condobservable} are in one-to-one correspondence with observables satisfying~\eqref{aaab}.
However, condition~\eqref{eq:condobservable} involves the so-called Oseledets-Lyapunov regularity function ($K$ in~\eqref{eq:condobservable}) which is notoriously hard to compute explicitly (as it is given as a supremum of a certain quantity).

Beyond temperedness, the study of properties of Oseledets-Lyapunov regularity functions has been initiated only recently by Simi\' c~\cite{S}. In particular, he shows that under appropriate regularity assumptions for a linear cocycle and mixing assumptions for a base space, one can achieve that these regularity functions belong to the $L^p$ space for sufficiently small values of $p>0$. Other relevant contributions to this area of research were obtained  more recently by Gou\"ezel and Stoyanov~\cite{GS}.

We would like to emphasize that conditions similar to~\eqref{eq:condobservable} have appeared earlier in the study of random dynamical systems. Indeed, an analogous requirement is present in the study of invariant manifolds (see~\cite[(7.3.2)]{Arnold}) and  linearization (see~\cite[Proposition 7.4.11]{Arnold}) of (nonlinear) random dynamical systems.\footnote{One should notice that the conditions in~\cite{Arnold} are stated in terms of the so-called Lyapunov norms. However, the random variable which measures the deviation of Lyapunov norms from the original norm is precisely Oseledets-Lyapunov regularity function.} We refer to~\cite[p.379]{Arnold} for a detailed discussion.  Moreover, a similar condition ensures persistence of nonuniform behaviour under small perturbations (see~\cite[(2.2)]{ZLZ}). Closer to the context of the present paper, the random variable measuring the speed of (exponential) decay of correlations for expanding on average cocycles (see~\cite[Main Theorem]{Buzzi}) is also not given explicitly. 
This indicates that the complexity of~\eqref{eq:condobservable} is not a consequence of our techniques but rather of intrinsic difficulties in treating nonuniformly hyperbolic dynamics.

Although this makes our results somewhat unsatisfactory, we still believe that the present approach offers a new insight on limit theorems for random systems which exhibit nonuniform decay of correlations. In particular, the above described connection with the Oseledets-Lyapunov regularity function sheds some light on the difficulty of providing explicit conditions for limit theorems in our setting.
Moreover,  our example from Appendix A, indicates that some condition similar to~\eqref{eq:condobservable} needs to be imposed for limit theorems to hold.

Recently, the second author~\cite{H} has made a significant breakthrough and obtained  explicit conditions under which limit theorems hold for certain classes of random systems exhibiting nonuniform decay of correlations. However, these classes do not include random expanding of average cocycles. In particular, it is still an open problem to describe explicitly the set of observables for which limit theorems hold in the framework of the simple example presented in Appendix A.

In conclusion, if our results do not aim to say the final word on the topic of limit theorems for expanding on average cocycles, they certainly represent a step forward.

	\section{A vector valued ASIP via Gou\"ezel's approach}\label{SecGou}
	\subsection{Preliminaries}
	We begin by recalling the setup from~\cite{Buzzi} (and also from~\cite{DDS}). Let $(X, \mathcal G)$ be a measurable space endowed with a probability measure $m$ and a notion of a variation $\var \colon L^1(X, m) \to [0, \infty]$ which satisfies
	the following conditions:
	\begin{enumerate}
		\item[(V1)] $\var (th)=\lvert t\rvert \var (h)$;
		\item[(V2)] $\var (g+h)\le \var (g)+\var (h)$;
		\item[(V3)] $\lVert h\rVert_{L^\infty} \le C_{\var}(\lVert h\rVert_1+\var (h))$ for some constant $1\le C_{\var}<\infty$;
		\item[(V4)] for any $C>0$, the set  $\{h\colon X \to \mathbb R: \lVert h\rVert_1+\var (h) \le C\}$ is $L^1(m)$-compact;
		\item[(V5)] $\var(\mathds 1)=0$, where $\mathds 1$ denotes the function equal to $1$ on $X$;
		\item[(V6)] $\{h \colon X \to \mathbb R_+: \lVert h\rVert_1=1 \ \text{and} \ \var (h)<\infty\}$ is $L^1(m)$-dense in
		$\{h\colon X \to \mathbb R_+: \lVert h\rVert_1=1\}$;
		\item[(V7)] for any $f\in L^1(X, m)$ such that $\essinf f>0$, we have \[\var(1/f) \le \frac{\var (f)}{(\essinf f)^2}.\]
		\item[(V8)] $\var (fg)\le \lVert f\rVert_{L^\infty}\cdot \var(g)+\lVert g\rVert_{L^\infty}\cdot \var(f)$;
		\item[(V9)] for $M>0$, $f\colon X \to [-M, M]$ measurable and  every $C^1$ function $h\colon [-M, M] \to \mathbb C$, we have
		$\var (h\circ f)\le \lVert h'\rVert_{L^\infty} \cdot \var(f)$.
	\end{enumerate} We define
	\[
	BV=BV(X,m)=\{g\in L^1(X, m): \var (g)<\infty \}.
	\]
	Then, $BV$ is a Banach space with respect to the norm
	\[
	\lVert g\rVert_{BV} =\lVert g\rVert_1+ \var (g).
	\]
	\begin{remark}
		Observe that (V3) and (V8) imply that 
		\begin{equation}\label{mc}
			\|fg\|_{BV} \le C_{var} \|f\|_{BV} \cdot \|g\|_{BV} \quad \text{for $f, g\in BV$.}
		\end{equation}
		
	\end{remark}
	
	\begin{remark}
		We observe that in \cite{Buzzi}, assumption (V5) is replaced by the weaker $\var(\mathds 1)<+\infty$. However, for the examples we have in mind, our stronger version is satisfied. In particular, (V5) implies that $\| \mathds 1\|_{BV}=1$.
	\end{remark}

	Let $(\Omega, \mathcal{F}, \mathbb P, \sigma)$ be a probability space and $\sigma \colon \Omega \to \Omega$  an invertible ergodic measure-preserving transformation. Let  $T_{\omega} \colon X \to X$, $\omega \in \Omega$ be a collection of non-singular  transformations (i.e.\ $m\circ T_\omega^{-1}\ll m$ for each $\omega$) acting   on $X$.  Each transformation $T_{\omega}$ induces the corresponding transfer operator $\mathcal L_{\omega}$ acting on $L^1(X, m)$ and  defined  by the following duality relation
	\[
	\int_X(\mathcal L_{\omega} \phi)\psi \, dm=\int_X\phi(\psi \circ T_{\omega})\, dm, \quad \phi \in L^1(X, m), \ \psi \in L^\infty(X, m).
	\]
	Thus, we obtain a cocycle of transfer operators  $(\Omega, \mathcal F, \mathbb P, \sigma, L^1(X, m), \mathcal L)$ that we denote by $\mathcal L=(\mathcal L_\omega)_{\omega \in \Omega}$. For $\omega \in \Omega$ and $n\in \mathbb N$, set
	\[
	\mathcal L_\omega^n:=\mathcal L_{\sigma^{n-1} \omega} \circ \ldots \circ \mathcal L_{\sigma \omega} \circ \mathcal L_\omega.
	\]
	\begin{definition}\label{good}
		A cocycle $\mathcal L=(\mathcal L_\omega)_{\omega \in \Omega}$ of transfer operators is said to be \emph{good} if the following conditions hold:
		\begin{itemize}
			\item $\Omega$ is a Borel subset of a separable, complete metric space and $\sigma$ is a homeomorphism. Moreover, $\mathcal L$ is $\mathbb P$-continuous, i.e. $\Omega$ can be written as a countable union of measurable sets such that $\omega \mapsto 
			\mathcal L_\omega$ is continuous on each of those sets;
			\item there exists a random variable $C\colon \Omega \to (0, +\infty)$ such that $\log C\in L^1(\Omega, \mathbb P)$ and
			\[
			\|\mathcal L_\omega h\|_{BV}\le C(\omega) \|h\|_{BV}, \quad \text{for $\mathbb P$-a.e. $\omega \in \Omega$ and $h\in BV$;}
			\]
			\item there exist $N\in \N$ and random variables $\alpha^N, K^N \colon \Omega \to (0, +\infty)$ such that 
			\[
			\int_\Omega \log  \alpha^N \, d\mathbb P <0, \quad \log K^N \in L^1(\Omega, \mathbb P)
			\]
			and, for $\mathbb P$-a.e. $\omega \in \Omega$ and $h\in BV$,
			\[
			\var(\mathcal L_\omega^N h) \le \alpha^N (\omega) \var(h)+ K^N(\omega) \|h \|_1;
			\]
			\item for each $a>0$ and $\mathbb P$-a.e. $\omega \in \Omega$, there exist random numbers $n_c(\omega)<+\infty$ and $\alpha_0(\omega), \alpha_1(\omega), \ldots$ such that for every $h\in \mathcal C_a$, 
			\begin{equation}\label{y}
				\essinf_x (\mathcal L_\omega^n h)(x) \ge \alpha_n\|h \|_1 \quad \text{for $n\ge n_c$,}
			\end{equation}
			where 
			\begin{equation}\label{cones}
				\mathcal C_a:=\{ h\in L^\infty (X,m): \text{$h\ge 0$ and $\var(h) \le a\|h\|_1$} \};
			\end{equation}
			\item $\log\left(\essinf_{x\in X}(\L_\omega \mathds 1)(x)\right)\in L^1(\Omega,\mathbb P)$. 
		\end{itemize}
	\end{definition}
	
	\begin{remark}\label{Rm}
\begin{itemize}
	\item The first requirement of Definition \ref{good}, $\mathds P$-continuity of the map $\omega\in\Omega\mapsto\L_\omega$, may be seen as restrictive: this is the price to pay to apply the Multiplicative Ergodic Theorem in a non-separable Banach space such as $BV$. We stress that the $\mathds P$-continuity of the map $\omega \mapsto \L_\omega$ holds whenever the map $\omega \mapsto T_\omega$ has  a countable range $\{T_1, T_2, \ldots \}$ and for each $j$, $\{T_\omega=T_j\}\in \mathcal F$.
	
	We note that when working with a separable Banach space $\B$, as is the case in Example \ref{ex:smoothexponaverage}, we can replace this requirement with the looser \emph{strong measurability}, i.e. measurability of $\omega\in\Omega\mapsto\L_\omega h$ for $h\in\B$.
	\item Definition~\ref{good} almost coincides with~\cite[Definition 13]{DDS}, the only difference being the addition of the last requirement in Definition~\ref{good}. 
	\item This log-integrability assumption may easily be checked on explicit examples (see e.g. the discussion in \cite[Remark 2.12]{Atnip}).
	\item Furthermore, this assumption implies a certain version of the ``random covering" similar to \eqref{y}. More precisely, 
		denoting by $\C_+$ the cone of non-negative function in $L^\infty(X)$, and by $\theta_+$ the projective Hilbert metric on this cone, and assuming that the transfer operator cocycle $\L$ satisfies \\$\log(\essinf\L_\omega\mathds 1)\in L^1(\Omega,\mathbb P)$, we have, for any $h\in\C_{+}$ such that $\theta_+(h,\mathds 1)\le R$ for some $0<R<+\infty$, the following: for $\mathbb P$-a.e $\omega\in\Omega$, there exists some random integer $n_c(\omega)$, \emph{non-random} positive numbers $(\alpha_n)_{n\ge 0}$ such that for any $n\ge n_c(\omega)$,
		\[\essinf_{x\in X}\L_\omega^n h\ge \alpha_n\|h\|_1.\]
		Recall (see e.g. \cite[Sec. 1.3]{Buzzi}) that $\theta_+(h,1)=\log\left(\frac{\esssup h}{\essinf h}\right)$, so that \[\theta_+(h,1)\le R\Longleftrightarrow \essinf h\ge e^{-R}\esssup h.\]
		Since the sequence $\essinf\L_{\omega}^n\mathds 1$ is super-multiplicative\footnote{we recall that a sequence of measurable functions $(f_n)_n$ on $\Omega$ is said to be super-multiplicative if $f_{n+m}(\omega)\ge f_m(\sigma^n \omega)\cdot f_n(\omega)$ for $\mathbb P$-a.e. $\omega \in \Omega$ and $m,n \in \mathbb N$.}, and by Birkhoff's ergodic theorem, we have
		\[\frac{1}{n}\log(\essinf\L_{\omega}^n\mathds 1)\ge \frac{1}{n}\sum_{k=0}^{n-1} \log (\essinf\L_{\sigma^k\omega}\mathds 1)\underset{n\to\infty}{\longrightarrow}\int_\Omega\log(\essinf\L_\omega\mathds 1)\, d\mathbb P(\omega).\]
		In particular, there is some integer $n_c:=n_c(\omega)<+\infty$, such that for $n\ge n_c$, $\essinf\L_{\omega}^n\mathds 1\ge e^{nI/2}$, where 
\[
I:=\int_\Omega\log(\essinf\L_\omega\mathds 1)\, d\mathbb P(\omega).
\]
		Hence, for any $h\in B_{\theta_+}(\mathds 1,R)$, we have
		\[\essinf\L_{\omega}^n h\ge (\essinf h) \cdot (\essinf\L_{\omega}^n\mathds 1)\ge e^{-R}e^{nI/2}\esssup h\ge \alpha_n\|h\|_1,\]
		with $\alpha_n:=e^{-R+nI/2}>0$, which is  non-random as announced.
\end{itemize}
	\end{remark}
	
	Let us now give examples of systems satisfying our requirements: the following is essentially taken from~\cite{Buzzi}.
	\begin{example}[Lasota-Yorke cocycles]\label{ex:good}
		Consider $X=[0,1]$, endowed with Lebesgue measure $m$ and the classical notion of variation $\var$. We say that $T:X\to X$ is a piecewise monotonic non-singular map (p.m.n.s map for short) if the following conditions hold:
		\begin{itemize}
			\item T is piecewise monotonic, i.e. there exists a subdivision $0=a_0<a_1<\dots<a_N=1$ such that for each $i\in\{0,\dots,N-1\}$, the restriction $T_i=T_{|(a_i,a_{i+1})}$ is monotonic (in particular it is a homeomorphism on its image).
			\item T is non-singular, i.e. there exists $|T'|:[0,1]\to\R_+$ such that for any measurable $E\subset (a_i,a_{i+1})$, $m(T(E))=\int_E|T'|dm$.
		\end{itemize}
		The intervals $(a_i,a_{i+1})_{i\in\{0,\dots,N-1\}}$ are called the intervals of $T$. We also set $N(T):=N$ and $\lambda(T):=\essinf_{[0,1]}|T'|$.
		
		We consider a family $(T_\omega)_{\omega\in\Omega}$ of random p.m.n.s as above, and such that $T:\Omega\times [0,1]\to[0,1],~(\omega,x)\mapsto T_\omega(x)$ is measurable.
		Denoting $N_\omega=N(T_\omega)$ and $\lambda_\omega=\lambda(T_\omega)$, we assume that
		\begin{itemize}
			\item The map $\omega\mapsto\left(\var\left(\frac{1}{|T'_\omega|}\right),N_\omega,\lambda_\omega,a_1,\dots,a_{N_\omega-1}\right)$ is measurable.
			\item We have the following expanding-on-average property:
			\begin{equation*}
				\lim_{K\to \infty} \int_\Omega\log \min \left(\lambda_\omega, K\right)~d\mathbb P(\omega) \in (0, +\infty]
			\end{equation*}
			\item The map $\log^+\left(\frac{N_\omega}{\lambda_\omega}\right)$ is integrable.
			\item The map $\log^+\left(\var\left(\frac{1}{|T_\omega'|}\right)\right)$ is integrable.
			\item $T_\omega$ is covering, i.e. for any interval $I\subset[0,1]$, there exists a random number $n_c(\omega)>0$ such that for any $n\ge n_c$, one has
			\begin{equation}\label{COVER}
				\essinf_{[0,1]} \mathcal L^{n}_\omega(\mathds1_I)>0.
			\end{equation}
			\item $\log\left(\essinf_{x\in X}(\L_\omega\mathds1)(x)\right)\in L^1(\Omega,\mathbb P)$.
		\end{itemize}
		We will call a cocycle satisfying the previous assumptions an \emph{expanding on average Lasota-Yorke cocycle}. For a countably-valued measurable family $(T_\omega)_{\omega\in\Omega}$ of expanding on average Lasota-Yorke cocycle, the associated cocycle 
		of transfer operators $(\mathcal L_\omega)_{\omega \in \Omega}$ is good (see~\cite{DDS}).
	\end{example}
	The following example can be fruitfully compared to a similar one by Kifer~\cite{K1}.
	\begin{example}\label{ex:smoothexponaverage}
		We consider $X=\mathbb S^1$, endowed with the Lebesgue measure  $m$ and the notion of variation given by $\var(\phi):=\int_X |\phi'|~dm=\|\phi'\|_{L^1}$. Notice that this notion of variation leads to define the $W^{1,1}(\mathbb S^1)$ Sobolev space instead of the space of bounded variation observables $BV$. We consider a measurable map  $T:\Omega\times X\to X$ such that  $T_\omega:=T(\omega,\cdot)$ is $C^r$, $r\ge 2$. In addition, we make the following assumptions:
		\begin{itemize}
			\item The map $\omega\in\Omega\mapsto \left(\int_X\frac{|T_\omega''|}{(T_\omega')^2} dm,\lambda_\omega\right)$ is measurable, where $\lambda_\omega=\inf_{[0, 1]}|T_\omega'|$.
			\item The following expanding on average property holds:
			\begin{equation}\label{hyp:exponaverage}
				\int_\Omega \log(\lambda_\omega)~d\mathds P(\omega)>0.
			\end{equation}
			\item The map $\log\left(\int_X\frac{|T_\omega''|}{(T_\omega')^2} dm\right)$ is $\mathbb P$-integrable.
			\item $\log\left(\essinf_{x\in X}(\L_\omega\mathds1)(x)\right)\in L^1(\Omega,\mathbb P)$.
		\end{itemize} 
		We call a family $(T_\omega)_{\omega\in\Omega}$ satisfying the previous assumptions a \emph{smooth expanding on average cocycle} (see also~\cite[Example 16]{DDS}). 
		\\We note that our expansion on average condition \eqref{hyp:exponaverage} implies that $\mathbb P$-a.s, $T_\omega$ has non-vanishing derivative, hence is a local diffeomorphism and a monotonic map of the circle. Furthermore, smooth expanding on average cocycles satisfy a stronger version of the random covering property (it implies \cite[Remark 0.1]{Buzzi} the one formulated in~\eqref{COVER}): for each non-trivial interval $I\subset X$, for $\mathbb P$-a.e $\omega\in\Omega$, there is a $n_c:=n_c(\omega,I)<\infty$ such that for all $n\ge n_c$,
		\[
		T_\omega^n(I)=X.
		\]
		To see this, first remark that by smoothness of the maps $T_\omega$, one has, for any interval $I\subset X$, and any $n\in\N$, that\footnote{Here we abuse notations, identifying the Lebesgue measure on $\mathbb S^1$,  the circle map $T_\omega^n$ and the small interval $I\subset\mathbb S^1$ with their lifted counterpart on $\R$.}
		\begin{align*}
			m(T_{\omega}^n(I))=\int_I|(T_\omega^n)'|dm\ge \lambda_\omega^n|I|,
		\end{align*}
		where  $\lambda_\omega^n:=\lambda_{\sigma^{n-1}\omega}\cdots\lambda_\omega$. Our expansion on average assumption \eqref{hyp:exponaverage} and Birkhoff's ergodic theorem insures that
		\[ \frac{1}{n}\log(\lambda_\omega^n)=\frac{1}{n}\sum_{j=0}^{n-1}\log(\lambda_{\sigma^j \omega})\underset{n\to\infty}{\longrightarrow}\Lambda:=\int_\Omega\log(\lambda_\omega)d\mathbb P(\omega)>0.\]
		Hence, for a.e $\omega\in\Omega$ we may choose measurably a $\tilde n_c(\omega)\in \N$, such that for $n\ge\tilde n_c(\omega)$ one has \[m(T_\omega^n(I))\ge |I| e^{n\Lambda/2}.\]
		Taking $n_c:=1+\max\left(\tilde n_c,-\dfrac{2}{\Lambda}\log(|I|)\right)$ gives the desired result.
		\\Finally, for a measurable, possibly uncountably valued family $(T_\omega)_{\omega\in\Omega}$ of smooth expanding on average cocycle, the associated cocycle of transfer operators $(\mathcal L_\omega)_{\omega \in \Omega}$ is strongly measurable on $W^{1,1}$: this follows from \cite[Prop. 4.11]{bomfim2016}, by arguing as in \cite[Proof of Prop. 5.2]{HC} (this fact was already noted in \cite[Prop. 24]{DGS}).
		\\Hence, the cocycle of transfer operators $(\mathcal L_\omega)_{\omega \in \Omega}$ is good in the sense of Remark \ref{Rm}, replacing the $\mathds P$-continuity requirement of Definition \ref{good} by a strong measurability one. We emphasize that, since in the present setting, we work on the separable Banach space $W^{1,1}$, the range of the measurable map $\omega\in\Omega\mapsto T_\omega$ may be infinite uncountable, in contrast with the previous example.
	\end{example}

Our abstract setup also covers multidimensional examples. The one we describe now is due to Buzzi \cite[Appendix B]{Buzzi}.

		\begin{example}[Multidimensional piecewise affine maps.]
	 Recall that a polytope in $\R^d$ is defined as the intersection of half-spaces. If $X\subset\R^d$, let $P$ be a finite collection of pairwise disjoints, open polytopes $A$ of $\R^d$, such that $Y=\bigcup_{A\in P} A$ is dense in $X$. We now let $f:Y\to X$ be such that for any $A\in P$, $f:A\to f(A)\subset X$ is the restriction of an affine map $f_A$ of $\R^d$: we say that $(X,P,f)$ is a piecewise affine map. We will also assume that each $f_A$ is invertible.
		\par\noindent We define the expansion rate of $f$, \[\lambda(f):=\inf_{x\in Y}\inf_{\|v\|=1}\|Df(x)\cdot v\|.\]
		We also recall that, given a polytope $A\subset\R^d$, we can define the $\epsilon$-multiplicity of its boundary $\partial A$ at $x\in X$, $\mathrm{mult}(\partial A,\epsilon,x)$, as the number of hyperplanes in $\partial A$ having non-empty intersection with $B(x,\epsilon)$ the ball of radius $\epsilon$ centered at $x$. We then set
		\begin{align*}
		\mathrm{mult}(\partial P,\epsilon)&:=\sup_{x\in X}\sum_{\underset{A\in P}{x\in\bar A}}\mathrm{mult}(\partial A,\epsilon,x)\\
		\mathrm{mult}(\partial P)&:=\lim_{\epsilon\to 0}\mathrm{mult}(\partial A,\epsilon).
		\end{align*}
		Finally we notice that there are some $\epsilon>0$ for which $\mathrm{mult}(\partial P,\epsilon)=\mathrm{mult}(\partial P)$. We denote by $\epsilon(\partial P)$ the supremum of such $\epsilon$.
		\\Given a probability space $(\Omega,\mathcal F,\mathds P)$, endowed as usual with an invertible, measure-preserving and ergodic self map $\sigma$, we consider countably-valued, measurable families of polytopes $(A_\omega)_{\omega\in\Omega}$ and affine maps $(f_{A_\omega})_{\omega\in\Omega}$ of $X\subset\R^d$: this data defines a cocycle of random piecewise affine map $(X,P_\omega,f_\omega)$, for which we assume:
		\begin{enumerate}
			\item For any $n\in\mathbb N$, the map $\omega\mapsto(\lambda(f^n_\omega),|P^n_\omega|,\mathrm{mult}(\partial P^n_\omega),\epsilon(\partial P^n_\omega))$ is measurable.
			\item The map $\frac{|P|}{\lambda}$ is $\log^+$ $\mathds P$-integrable.
			\item The following expansion on average condition holds:
			\begin{equation*}
			\Lambda:=\lim_{n\to\infty}\lim_{K\to\infty}\int_\Omega\frac{1}{n}\log\min\left(\frac{\lambda^n_\omega}{\mathrm{mult}(P^n_\omega)},K\right)d\mathds P>0.
			\end{equation*}
			\item The following random covering condition holds: 
			\\For any ball $B\subset X$, $\mathds P$-a.e $\omega\in\Omega$ there is a $n_c:=n_c(\omega,B)$ such that $f_\omega^n(B)=X$ (modulo a null set for Lebesgue measure) for $n\ge n_c$.
			\item $\log(\essinf_{x\in X}\L_\omega\mathds 1(x)) \in L^1(\Omega,\mathds P)$. 
		\end{enumerate}
		Under the previous assumptions, and for the notion of variation on $X$ given by\footnote{This notion of variation fulfills conditions (V1)-(V9): we refer to~\cite[Section 2.2.]{DFGTV2} for details.}
		\begin{equation}\label{def:HDvar}
		\var (f)=\sup_{0<\epsilon \le \epsilon_0}\frac{1}{\epsilon^\alpha}\int_{\R^d}osc (f,  B_\epsilon (x)))\, dm(x),
		\end{equation}
		where
		\[
		osc (f, B_\epsilon (x))=\esssup_{x_1, x_2 \in B_\epsilon (x)}\lvert f(x_1)-f(x_2)\rvert.
		\]
	 it is established in \cite[Prop B.1]{Buzzi} that a random piecewise affine map has a good random transfer operator, in the sense of \cite[Def.1.1]{Buzzi}. Together with the assumption that this map is countably valued, this shows that the associated transfer operator cocycle is good in the sense of Definition \ref{good}.
	\end{example}
	
	We recall the  notion of a tempered random variable.
	\begin{definition}
		We say that a measurable map $K\colon \Omega \to (0, +\infty)$ is \emph{tempered} if
		\[
		\lim_{n\to \pm \infty} \frac 1n \log K(\sigma^n \omega)=0, \quad \text{for $\mathbb P$-a.e. $\omega \in \Omega$.}
		\]
	\end{definition}
	We will need the following classical result (see~\cite[Proposition 4.3.3.]{Arnold}).
	\begin{proposition}\label{PA}
		Let $K \colon \Omega \to (0, +\infty)$ be a tempered random variable. For each $\epsilon >0$, there exists a tempered random variable $K_\epsilon \colon \Omega \to (1, +\infty)$ such that
		\[
		\frac{1}{K_\epsilon (\omega)} \le K(\omega) \le K_\epsilon (\omega) \quad \text{and} \quad K_\epsilon(\omega)e^{-\epsilon |n|} \le K_\epsilon (\sigma^n \omega) \le K_\epsilon (\omega) e^{\epsilon |n|},
		\]
		for $\mathbb P$-a.e. $\omega \in \Omega$ and $n \in \mathbb Z$.
	\end{proposition}
	
	\subsection{Statement of the main result}
	We are now in a position to state the main result of our paper.  By $x\cdot y$ we will denote the scalar product of $x, y\in \mathbb C^d$.
	\begin{theorem}\label{TM}
		Let $\mathcal L=(\mathcal L_\omega)_{\omega \in \Omega}$ be a good cocycle of transfer operators. Moreover, take $d\in \N$ and let 
		$\psi =(\psi^1, \ldots, \psi^d) \colon \Omega \times X \to \R^d$ be a measurable map such that the following conditions hold:
		\begin{itemize}
			\item $\psi_\omega^i:=\psi^i(\omega, \cdot ) \in BV$ for $\omega \in \Omega$ and $1\le i \le d$;
			\item for $1\le i \le d$, we have that \begin{equation}\label{eq:condobservable}
				\esssup_{\omega \in \Omega} \bigg (K( \omega)  \|\psi_\omega^i\|_{BV} \bigg )<+\infty, 
			\end{equation}
			where $K\colon \Omega \to [1, +\infty)$ is a tempered random variable  given by Lemma~\ref{PRO};
			\item for $1\le i \le d$ and  $\mathbb P$-a.e. $\omega \in \Omega$,
			\begin{equation}\label{center}
				\int_X \psi_\omega^i \, d\mu_\omega=0,
			\end{equation}
			where $\mu_\omega$, $\omega \in  \Omega$ are probability measures on $X$ as in the statement of Corollary~\ref{Cor} (namely $\{\mu_\omega\}$ is the unique family of absolutely continuous equivariant measures).
		\end{itemize}
		Then, we have the following:
		\begin{enumerate}
			\item there exists a positive semi-definite $d\times d$ matrix $\Sigma^2$ such that for $\mathbb P$-a.e. $\omega \in \Omega$ we have 
			\[
			\lim_{n\to\infty}\frac 1n \int_X \big(S_n \psi (\omega, \cdot)\big)^2\, d\mu_\omega=\Sigma^2,
			\]
			where
			\begin{equation}\label{RBS}
				S_n \psi (\omega, x)=\sum_{i=0}^{n-1}\psi (\sigma^i \omega, T_\omega^i(x)), \quad (\omega, x)\in \Omega \times X.
			\end{equation}
			Moreover, $\Sigma^2$ is not positive definite if and only if there exist  $v\in\mathbb R^d\setminus \{0\}$ and an $\mathbb R$-valued measurable function $r$ on $\Omega\times X$ such that $\mathbb P$-a.s $r(\omega,\cdot)\in BV$, $\esssup_{\omega \in \Omega}\|r(\omega,\cdot)\|_{BV}<\infty$ and
			\begin{equation}\label{Cob}
				v\cdot \psi=r-r\circ \tau,\,\,\,\mu-\text{a.e,}
			\end{equation} 
			where $\tau$ and $\mu$ are given by~\eqref{SP} and~\eqref{mu}, respectively. 
			\item Suppose that $\Sigma^2$  is positive definite. Then, for $\mathbb P$-a.e. $\omega \in\Omega$ and every $\delta >0$, there exists a coupling between  $\{\psi_{\sigma^n\omega}\circ T_\omega^{n}:n\geq 0\}$, considered as a sequence of random variables on $(X,\mathcal B,\mu_\omega)$,
			and a  sequence $(Z_k)_k$ of independent centered (i.e. of zero mean) Gaussian random vectors
			such that
			\[
			\bigg{\lvert} \sum_{i=0}^{n-1}\psi (\sigma^i\omega, \cdot)\circ T_\omega^{i}-\sum_{i=1}^n Z_i \bigg{\rvert} =O(n^{1/4+\delta}),\quad\text{almost-surely}.
			\]
			Moreover, there exists a constant $C=C_\delta(\omega)>0$ so that for every $n\geq1$,
			\begin{equation*}
				\left\|\sum_{i=0}^{n-1}\psi (\sigma^i\omega, \cdot)\circ T_\omega^{i}-\sum_{i=1}^n Z_i\right\|_{L^2}\leq Cn^{1/4+\delta}.
			\end{equation*}
			Finally, there is a constant $C'=C'_\delta(\omega)>0$ so that for every unit vector $v\in \mathbb R^d$,
			\[
			\left|\left\|\sum_{i=1}^n Z_i\cdot v\right\|_{L^2}^2-\left\|\sum_{i=0}^{n-1} \psi (\sigma^i\omega, \cdot)\circ T_\omega^{i}\cdot v\right\|_{L^2}^2\right|\leq C'n^{1/2+\delta}.
			\]
		\end{enumerate}
	\end{theorem}
	
	\begin{remark}
		Let us comment on the statement of Theorem~\ref{TM}:
		\begin{itemize}
		\item Reasoning as in~\cite[Remark 34]{DDS}, it is easily seen that in the setting of~\cite{DFGTV1}, $K$ is constant. Hence, \eqref{eq:condobservable} is equivalent to
		\begin{equation}\label{aaab}
		\esssup_{\omega \in \Omega}   \|\psi_\omega^i\|_{BV} <+\infty, \quad 1\le i \le d.
		\end{equation}
		Therefore, in the setting of~\cite{DFGTV1}, Theorem~\ref{TM} reduces to~\cite[Theorem 4.18]{DH3}.
		\item It is possible to construct observables satisfying assumption \eqref{eq:condobservable} by following, for each scalar map $\psi_\omega^i,~i\in\{1,\dots,d\},$ the procedure described in \cite[Example 35]{DDS}.
		\end{itemize}
\end{remark}

\begin{remark}
		 In \cite{DH1}  a version of Theorem \ref{TM} for the random expanding maps from \cite[Ch. 5]{HK} (see also \cite{MSU}) was established. While this was obtained only for H\"older continuous observables, we stress that  these maps are not absolutely continuous with respect to a given reference measure, and so the setup of \cite{DH1} is not included in the setup of the present paper. As discussed in Section \ref{Sec1}, the results from~\cite{DH1} were obtained by passing to an induced system which exhibits uniform decay of correlations.
		Thus, a completely different type of assumptions on the observables were needed. However, we believe that the arguments in the present paper also yield Theorem \ref{TM} in the setup of \cite{DH1}. The main obstacle is to establish \eqref{est11} and \eqref{est22} with a tempered random variable $\tilde D(\omega)$ (and not just with a one which is finite a.e.). Under appropriate log integrability conditions, when the maps $T_\omega$ in \cite{DH1} act on the same space this can be achieved by an application of Oseledets theorem, while in the case of maps $T_\omega:\mathcal E_\omega\to\mathcal E_{\sigma\omega}$ between random measurable spaces $\{\mathcal E_\omega\}$,  this  can be achieved by using a recent version of Oseledets theorem for Banach fields  established in~\cite{VR}\footnote{To apply the latter, it seems that we need to impose  stronger integrability assumptions than~\cite{DH1}. More precisely, we believe that the  random variable $Q_\omega$ defined in \cite[(2.16)]{MSU} needs to be integrable}.

	\end{remark}

	\subsection{Proof of Theorem~\ref{TM}: behaviour of the cocycle of normalized transfer operators}
	\begin{theorem}\label{T}
		Let $\mathcal L=(\mathcal L_\omega)_{\omega \in \Omega}$ be a good cocycle of transfer operators. Then, the following holds:
		\begin{itemize}
			\item there exists an essentially unique measurable family $(v_\omega^0)_{\omega \in \Omega}\subset BV$ such that $v_\omega \ge 0$, $\int_X v_\omega^0 \, dm=1$ and
			\[
			\mathcal L_\omega v_\omega^0=v_{\sigma \omega}^0, \quad \text{for $\mathbb P$-a.e. $\omega \in \Omega$;}
			\]
			\item there is a random variable $\ell:\Omega\to(0,+\infty)$ such that for $\mathbb P$-a.e. $\omega \in \Omega$, 
			\begin{equation}\label{ell}
				v_\omega^0 \ge \ell(\omega) \quad \text{$m$-a.e.;}
			\end{equation}
			\item for $\mathbb P$-a.e. $\omega \in \Omega$, 
			\begin{equation}\label{split}
				BV=span\{v_\omega^0\} \oplus BV^0,
			\end{equation}
			where \[ BV^0=\bigg \{ h\in BV: \int_X h\, dm=0 \bigg \}; \]
			\item $\omega \mapsto \|v_\omega^0\|_{BV}$ is tempered;
			\item there exist $\lambda >0$ and  for each $\epsilon >0$, a  tempered random variable $D=D_\epsilon \colon \Omega \to (0, +\infty)$ such that for $\mathbb P$-a.e. $\omega \in \Omega$ and $n\in \mathbb N$,
			\begin{equation}\label{est1x}
				\| \mathcal L_\omega^n \Pi(\omega) \|_{BV} \le D(\omega)e^{-\lambda n}
			\end{equation}
			and 
			\begin{equation}\label{est2x}
				\| \mathcal L_\omega^n(\Id- \Pi(\omega)) \|_{BV} \le D(\omega)e^{\epsilon n},
			\end{equation}
			where $\Pi(\omega) \colon BV \to BV^0$ is a projection associated to the splitting~\eqref{split}.
		\end{itemize}
	\end{theorem}
	
	\begin{proof}
		The first assertion of the theorem is established in~\cite{Buzzi}, while the third assertion is proved in~\cite[Proposition 24]{DDS}. Moreover, the last two statements of the theorem follow from~\cite[Proposition 23]{DDS} and~\cite[Proposition 28]{DDS} respectively. 
		
		Thus, it only remains to establish the second assertion of the theorem: first, we remark that when $\omega$ is good\footnote{we note that several parameters $a, R, B_*, \alpha_*$ associated with this notion  will be used in the sequel (where $\varepsilon$ from \cite[Definition 2.4]{Buzzi} is a sufficiently small fixed number)} in the sense of \cite[Definition 2.4]{Buzzi}, one has $v_\omega^0\in \mathcal C_a$ for some $a>0$, where $\mathcal C_a$ is given by~\eqref{cones}. Indeed, we may write, by \cite[Lemma 2.1]{Buzzi}	
		\begin{equation*}
			\var(\L_{\sigma^{-n}\omega}^n \mathds1)\le C_0(\omega)\var(\mathds 1)+C_0(\omega)\int_X \mathds1 dm=C_0(\omega),
		\end{equation*}	
		where $C_0$ is some a.e finite function. Moreover, by~\eqref{est1x}  and Proposition~\ref{PA} we have that \[ \|\L_{\sigma^{-n}\omega}^n \mathds 1-v_{\omega}^0\|_{BV}\le D(\sigma^{-n} \omega)e^{-\lambda n}\le D_{\lambda/2}(\omega)e^{-\frac{\lambda}{2} n}. \]  Taking the limit as $n\to\infty$, we obtain that \[\var(v_\omega^0)\le C_0(\omega)\le B_\ast,\] if $\omega$ is good. Since $a\ge 6B_\ast$, we obtain that $v_\omega^0\in \mathcal C_{a/6}\subset \mathcal C_a$ for good $\omega$.
		In particular, we get that
		\[\essinf v_{\sigma^R\omega}^0=\essinf\L_\omega^Rv_\omega^0\ge \alpha_\ast. \]
		Hence, for every $\omega\in\sigma^{-R}(\Omega_\ast)=:\Omega^1_+$, where $\Omega_\ast$ is the set of good parameters, $\essinf v_\omega^0\ge \alpha_\ast$. We note that $\mathbb P(\Omega^1_+)=1-\frac{\epsilon}{4}>0$ by \cite[Lemma 2.6]{Buzzi} and the measure preserving property of $\sigma$. 
		\\Let us consider now the set $\Omega^2_+:=\{\omega\in\Omega: \essinf\L_\omega \mathds1>0\}$. 
		Our $\log$-integrability assumption on $\essinf\L_\omega \mathds1$ entails that this set has full measure, and up to replacing it by $\bigcap_{k\in\mathbb Z}\sigma^k(\Omega^2_+)$, we can assume that it is $\sigma$-invariant. Hence, for a.e $\omega\in\Omega$ and $n\in \N$ we have that \[ \essinf\L^n_{\sigma^{-n}\omega}\mathds1\ge\essinf\L_{\sigma^{-1}\omega}\mathds1\cdots\essinf\L_{\sigma^{-n}\omega}\mathds1>0.\]
		We may now introduce the first hitting time of the positive measure set $\Omega_+:=\Omega^1_+\cap\Omega^2_+$, i.e. we set, for $\omega\in\Omega$
		\[\tilde n_\omega:=\inf\{n\in\N: \sigma^{-n}\omega \in\Omega_+\}.\]
		Therefore, we have that 
		\[
		\essinf v_{\omega}^0= \essinf \L_{\sigma^{-\tilde n_\omega}\omega}^{\tilde n_\omega}v_{\sigma^{-\tilde n_\omega}\omega}\ge\alpha_\ast \cdot \essinf \L_{\sigma^{-\tilde n_\omega}\omega}^{\tilde n_\omega}\mathds1>0
		\]
		Hence, \eqref{ell} holds with 
		\[ \ell(\omega):= \alpha_\ast \cdot \essinf\L_{\sigma^{-\tilde n_\omega}\omega}^{\tilde n_\omega}\mathds1>0, \quad \omega \in \Omega. \]
	\end{proof}
	
	\begin{remark}\label{newrmk}
	Using the language of the multiplicative ergodic theory (see~\cite[Section 2]{DDS}), the estimates~\eqref{est1x} and~\eqref{est2x} mean that for the cocycle $\mathcal L=(\mathcal L_\omega)_{\omega \in \Omega}$,  the separation between the Oseledets subspace corresponding to the largest Lyapunov exponent (which is zero)   and the sum of  Oseledets subspaces corresponding to all other Lyapunov exponents is measured by a tempered random variable. We stress that this is a general fact that holds for arbitrary cocycles (of not necessarily transfer operators). We refer to~\cite[Propostion 3.2]{BD} for  a precise formulation.
	
	\end{remark}
	
	\begin{cor}\label{Cor}
		Let $\mathcal L=(\mathcal L_\omega)_{\omega \in \Omega}$ be a good cocycle of transfer operators. 
		
		Then, the following holds:
		\begin{itemize}
			\item If $(v_\omega^0)_{\omega \in \Omega}\subset BV$  is given by Theorem~\ref{T}, then 
			\begin{equation}\label{a}
				\omega \mapsto \| 1/v_\omega^0\|_{BV} \  \text{is tempered.}
			\end{equation}
			\item For $\mathbb P$-a.e. $\omega \in \Omega$, 
			\begin{equation}\label{split2}
				BV=span\{\mathds 1\} \oplus BV_\omega^0,
			\end{equation}
			where 
			\[
			BV_\omega^0=\bigg \{ h\in BV: \int_X h\, d\mu_\omega=0 \bigg \},
			\]
			and $d\mu_\omega=v_\omega^0 dm$, $\omega \in \Omega$;
			\item  there exist $\lambda' >0$ and  a  tempered random variable $\tilde D\colon \Omega \to (0, +\infty)$ such that for $\mathbb P$-a.e. $\omega \in \Omega$ and $n\in \mathbb N$,
			\begin{equation}\label{est11}
				\| L_\omega^n \tilde \Pi(\omega) \|_{BV} \le \tilde  D(\omega)e^{-\lambda' n}
			\end{equation}
			\begin{equation}\label{est22}
				\| L_\omega^n(\Id- \tilde \Pi(\omega)) \|_{BV} \le \tilde  D(\omega),
			\end{equation}
			where $\tilde \Pi(\omega) \colon BV \to BV_\omega^0$ is a projection associated to the splitting~\eqref{split2}, and 
			\[
			L_\omega^n h =\mathcal L_\omega^n(hv_\omega^0) /v_{\sigma^n \omega}^0, \quad h\in BV, \ n\in \N.
			\]
		\end{itemize}
	\end{cor}
	
	\begin{proof}
		We first establish $\eqref{a}$. Given that
		\begin{eqnarray}
			\frac{\var(v_\omega^0)}{\esssup(v_\omega^0)^2}\le\var\left(\frac{1}{v_\omega^0}\right)\le \frac{\var(v_\omega^0)}{\essinf(v_\omega^0)^2},
		\end{eqnarray}
		it is enough to show that $\essinf(v_\omega^0)$ is tempered. Indeed, $\var(v_\omega^0)$ is tempered by Theorem \ref{T}, which implies that $\esssup v_\omega^0$ is by (V3). We have, thanks to $v_{\sigma^n\omega}^0=\L_\omega^nv_\omega^0$:
		\[\frac{1}{n}\log(\essinf v_{\sigma^n\omega}^0)\ge \frac{1}{n}\log(\essinf v_\omega^0)+\frac{1}{n}\log(\essinf\L_{\omega}^n \mathds1).\]
		By \eqref{ell}, the first term in the R.H.S. above goes to $0$ as $n\to\infty$, and for the second term, we notice that by the last item of Definition \ref{good}, super--multiplicativity of the sequence $(\essinf \L_\omega^n 1)_{n\ge 0}$ and Birkhoff's ergodic theorem, one has
		\[\frac{1}{n}\log(\essinf \L_\omega^n \mathds1)\ge \frac{1}{n}\sum_{k=0}^{n-1}\log\left(\essinf \L_{\sigma^k\omega} \mathds1\right)\underset{n\to\infty}{\longrightarrow}\int_\Omega\log(\essinf\L_\omega\mathds1) d\mathbb P(\omega),\]
		for $\mathbb P$-a.e. $\omega \in  \Omega$. 
		Hence it must follow that
		\[
		\begin{split}
			\liminf_{n\to\infty}\frac{1}{n}\log(\essinf v_{\sigma^n\omega}^0)&\ge \liminf_{n\to\infty}\frac{1}{n}\log(\essinf \L_\omega^n \mathds1) \\
			&\ge \int_\Omega\log(\essinf\L_\omega \mathds1) d\mathbb P(\omega).
		\end{split}
		\]
		In particular, this last bound is finite: by Tanny's theorem\footnote{Note that we use a version of Tanny's theorem for non-positive functions, whereas the ``standard" version concerns non-negative ones.} \cite[Theorem C.1]{GTQ}, this entails that $\lim_{n\to\infty}\frac{1}{n}\log(\essinf v_{\sigma^n\omega}^0)=0$, i.e. that  $\essinf(v_\omega^0)$ is tempered. Note that in the application of Tanny's theorem we have used that $\essinf v_\omega^0\leq 1$, which holds since  $\int_X v_\omega^0dm(\omega)=1$ and $v_\omega^0\geq0$.
		
		Next, we observe that for each $h\in BV$ and $\mathbb P$-a.e. $\omega \in \Omega$, 
		\[
		h=\bigg ( \int_X h\, d\mu_\omega \bigg)\mathds1+ \bigg (h-\bigg ( \int_X h\, d\mu_\omega \bigg)\mathds1 \bigg )\in span \{\mathds1\} +BV_\omega^0.
		\]
		On the other hand, clearly we have that $span \{\mathds1\} \cap BV_\omega^0=\{ 0\}$. Thus, \eqref{split2} holds for $\mathbb P$-a.e. $\omega \in \Omega$. Moreover,
		\begin{equation}\label{Pi}
			\tilde \Pi(\omega)h=h-\bigg ( \int_X h\, d\mu_\omega \bigg)\mathds1, \quad  \text{for $\mathbb P$-a.e. $\omega \in \Omega$ and $h\in BV$.}
		\end{equation}
		Since $\omega \mapsto \|v_\omega^0 \|_{BV}$ is tempered, it follows that $\omega \mapsto \|\tilde \Pi(\omega) \|_{BV}$ is tempered.
		
		Take an arbitrary $\epsilon >0$ and let $\lambda >0$ and $D=D_{\epsilon} \colon \Omega \to (0, +\infty)$ be given by Theorem~\ref{T}. Since $\omega \mapsto \|1/v_\omega^0\|_{BV}$ is tempered, by Proposition~\ref{PA} there exists a tempered random variable 
		$K\colon \Omega \to (0, +\infty)$ such that 
		\begin{equation}\label{9}
			\| 1/v_\omega^0 \|_{BV} \le K(\omega) \quad \text{and} \quad K(\omega)e^{-\lambda/2 |n|} \le K (\sigma^n \omega) \le K (\omega) e^{\lambda/2 |n|},
		\end{equation}
		for $\mathbb P$-a.e. $\omega \in \Omega$ and $n\in \Z$.
		By~\eqref{mc}, \eqref{est1x} and~\eqref{9}, it follows that
		\[
		\begin{split}
			\|L_\omega^n h\|_{BV}&=\| \mathcal L_\omega^n(hv_\omega^0) /v_{\sigma^n \omega}^0\|_{BV} \\
			&\le C_{var} \| \mathcal L_\omega^n(hv_\omega^0)\|_{BV} \cdot \| 1/v_{\sigma^n \omega}^0\|_{BV} \\
			&\le C_{var} D(\omega)e^{-\lambda n} \| hv_\omega^0\|_{BV}\cdot \| 1/v_{\sigma^n \omega}^0\|_{BV} \\
			&\le C_{var}^2  D(\omega)e^{-\lambda n} \| h\|_{BV} \cdot \|v_\omega^0 \|_{BV} \cdot \| 1/v_{\sigma^n \omega}^0\|_{BV} \\
			&\le C_{var}^2  D(\omega)K(\sigma^n \omega) e^{-\lambda n} \|v_\omega^0 \|_{BV}\cdot \| h\|_{BV} \\
			&\le C_{var}^2 D(\omega)K(\omega) e^{-\frac{\lambda}{2}  n} \|v_\omega^0 \|_{BV}\cdot \| h\|_{BV}, \\
		\end{split}
		\]
		for $\mathbb P$-a.e. $\omega \in \Omega$, $h\in BV$ such that $\int_X h\, d\mu_\omega=0$ and $n\in \mathbb N$. Thus, \eqref{est11} holds with $\lambda'=\lambda /2>0$ and 
		\begin{equation}\label{D1}
			\tilde D(\omega)=C_{var}^2 D(\omega)K(\omega)\| \tilde \Pi(\omega)\|_{BV} \cdot \|v_\omega^0 \|_{BV} \quad \omega \in \Omega,
		\end{equation}
		which is a  tempered random variable. 
		
		On the other hand, \eqref{Pi} together with the simple observations that $L_\omega^n\mathds 1=\mathds 1$ and $\|\mathds 1\|_{BV}=1$ implies that 
		\[
		\begin{split}
			\|L_\omega^n (\Id- \tilde \Pi (\omega))h\|_{BV}&= \bigg | \int_X h\, d\mu_\omega \bigg  | \cdot \| L_\omega^n \mathds 1\|_{BV} \\
			&\le \| h\|_\infty\\
			&\le C_{var}\| h \|_{BV}
		\end{split}
		\]
		for $\mathbb  P$-a.e. $\omega \in \Omega$ and $n\in \N$. Thus, \eqref{est22} holds with 
		\begin{equation}\label{D2}
			\tilde D(\omega)=C_{var}
		\end{equation}
		which is  constant and thus also tempered. Hence, \eqref{est11} and~\eqref{est22} hold with $\tilde D$ being the maximum of the expressions in~\eqref{D1} and~\eqref{D2}. This completes the proof. 
	\end{proof}

	\subsection{Proof of Theorem~\ref{TM}: Adapted norms}

	
	\begin{lemma}\label{PRO}
		Let $\mathcal L=(\mathcal L_\omega)_{\omega \in \Omega}$ be a good cocycle of transfer operators. Then, there is a family $\| \cdot \|_\omega$, $\omega \in \Omega$ of norms on $BV$ with the following properties:
		
		\begin{enumerate}
			\item There exists a tempered random variable $K\colon \Omega \to [1, +\infty)$ such that
			\begin{equation}\label{ln1}
				\| \phi \|_{BV} \le \| \phi \|_\omega \le K(\omega) \| \phi \|_{BV} \quad \text{for $\mathbb P$-a.e. $\omega \in \Omega$ and $\phi \in BV$;}
			\end{equation}
			In particular, $\|\cdot\|_\omega$ is a complete norm.
			\vskip0.1cm
			\item For $\mathbb P$-a.e. $\omega \in \Omega$, $\phi \in BV$ and $n\in \mathbb N$,
			\begin{equation}\label{ln2}
				\| L_\omega^n \tilde \Pi(\omega)\phi \|_{\sigma^n \omega}\le e^{-\lambda' n}\| \phi \|_\omega;
			\end{equation}
			\item For $\mathbb P$-a.e. $\omega \in \Omega$ and $\phi \in BV$,
			\begin{equation}\label{ln3}
				\bigg |\int_X \phi \, d\mu_\omega \bigg | \le \| \phi \|_\omega;
			\end{equation}
			\item for $\mathbb P$-a.e. $\omega \in \Omega$ and $\phi \in BV$,
			\begin{equation}\label{ln4}
				\| L_\omega \phi \|_{\sigma \omega} \le \| \phi \|_\omega;
			\end{equation}
			\item we have that 
			\begin{equation}\label{ln5}
				\|\mathds 1\|_\omega =1, \quad \text{for $\mathbb P$-a.e. $\omega \in \Omega$.}
			\end{equation}
		\end{enumerate}
	\end{lemma}
	
	\begin{proof}
		By using the same notation as in the statement of Corollary~\ref{Cor}, we set 
		\begin{equation}\label{AN}
			\begin{split}
				\| \phi \|_\omega &=\sup_{n\in \N} (\| L_\omega^n \tilde \Pi(\omega)\phi \|_{BV} e^{\lambda' n})+\bigg | \int_X\phi \, d\mu_\omega \bigg |,
			\end{split}
		\end{equation}
		for $\phi \in BV$ and $\mathbb P$-a.e. $\omega \in \Omega$.
		
		We begin by observing that it follows from~\eqref{est11}  and the simple observation  $\| \cdot \|_{L^1(\mu_\omega)}\le C_{var} \| \cdot \|_{BV}$ that 
		\[
		\| \phi \|_\omega \le ( \tilde D(\omega)+C_{var})\| \phi \|_{BV},
		\]
		for $\mathbb P$-a.e. $\omega \in \Omega$ and $\phi \in BV$. On the other hand,
		\[
		\begin{split}
			\| \phi \|_\omega & \ge \| \tilde{\Pi} (\omega) \phi \|_{BV}+\bigg | \int_X \phi \, d\mu_\omega \bigg | \\
			&\ge \| \tilde{\Pi} (\omega) \phi \|_{BV}+\|(\Id-\tilde \Pi(\omega))\phi \|_{BV} \\
			&\ge \| \phi \|_{BV},
		\end{split}
		\]
		for $\mathbb P$-a.e. $\omega \in \Omega$ and $\phi \in BV$. Hence, \eqref{ln1} holds with
		\begin{equation}\label{K-exp}
		K(\omega)= \tilde D(\omega)+C_{var}, \quad \omega \in \Omega.
		\end{equation}
		Moreover, since $\int_X\tilde \Pi(\omega)\phi d\mu_\omega=0$ we have that 
		\[
		\begin{split}
			\| L_\omega^n \tilde \Pi(\omega)\phi \|_{\sigma^n \omega} &=\sup_{m\in \N}( \| L_{\sigma^n \omega}^m \tilde \Pi(\sigma^n \omega) L_\omega^n \phi \|_{BV} e^{\lambda' m}) \\
			&=\sup_{m\in \mathbb N}(\| L_\omega^{n+m}\tilde \Pi(\omega)\phi \|_{BV}e^{\lambda' m}) \\
			&=e^{-\lambda' n}\sup_{m\in \mathbb N}(\| L_\omega^{n+m}\tilde \Pi(\omega)\phi \|_{BV}e^{\lambda' (m+n)}) \\
			&\le e^{-\lambda' n}\sup_{m\in \mathbb N}(\| L_\omega^{m}\tilde \Pi(\omega)\phi \|_{BV}e^{\lambda' m}) \\
			&\le e^{-\lambda'n} \| \phi \|_\omega,
		\end{split}
		\]
		for $\mathbb P$-a.e. $\omega \in \Omega$, $n\in \mathbb N$ and $\phi \in BV$. We conclude that~\eqref{ln2} holds. 
		Furthermore, \eqref{ln3} follows readily from~\eqref{AN}.
		In addition, we have that 
		\[
		\begin{split}
			\| L_\omega \phi \|_{\sigma \omega} &=\sup_{n\in \N} (\| L_{\sigma \omega}^n \tilde \Pi(\sigma \omega)L_\omega \phi \|_{BV} e^{\lambda' n})+ \bigg | \int_X L_\omega \phi \, d\mu_{\sigma \omega} \bigg | \\
			&= \sup_{n\in \N} (\| L_{ \omega}^{n+1} \tilde \Pi(\omega) \phi \|_{BV} e^{\lambda' n})+ \bigg | \int_X \phi \, d\mu_\omega \bigg | \\
			&=e^{-\lambda'} \sup_{n\in \N} (\| L_{ \omega}^{n+1} \tilde \Pi(\omega) \phi \|_{BV} e^{\lambda' ( n+1)})+\bigg | \int_X \phi \, d\mu_\omega \bigg |  \\
			&\le \| \phi \|_\omega,
		\end{split}
		\]
		for $\mathbb P$-a.e. $\omega \in \Omega$ and $h\in BV$. Thus, \eqref{ln4} holds. Finally, \eqref{ln5} follows directly from~\eqref{AN}.
	\end{proof}
	
	\begin{remark}\label{norms}
		In~\cite[Section 3.1]{DDS}, we introduced a similar class of norms, adapted to the original cocycle of transfer operators $(\mathcal L_\omega)_{\omega \in \Omega}$. On the other hand, in Lemma~\ref{PRO} we construct norms adapted to the associated cocycle of normalized transfer operators 
		$(L_\omega)_{\omega \in \Omega}$.
		
		To the best of our understanding, in order to construct appropriate adapted norms for the cocycle $(L_\omega)_{\omega \in \Omega}$, one needs the additional requirement in Definition~\ref{good}.
		
		Finally, we note that adapted norms given by Lemma~\ref{PRO} have a simpler form than those constructed in the proof of~\cite[Proposition 30]{DDS}. The reason is that the top Oseledets space (see~\cite[Section 2.1]{DDS}) of the cocycle 
		$(L_\omega)_{\omega \in \Omega}$ is spanned by $\mathds 1$ and $L_\omega \mathds 1=\mathds 1$. Consequently, the cocycle $(L_\omega)_{\omega \in \Omega}$  does not exhibit any growth along the associated top Oseledets space.
	\end{remark}
	
	We also describe the construction of dual adapted norms. 
	
	\begin{lemma}\label{LL}
		Let $\mathcal L=(\mathcal L_\omega)_{\omega \in \Omega}$ be a good cocycle of transfer operators. Then, there is a family $\| \cdot \|_\omega^*$, $\omega \in \Omega$ of norms on $BV^*$ with the following properties:
		\begin{enumerate}
			\item For $\mathbb P$-a.e $\omega\in\Omega$ and $\ell\in BV^*$,
			\begin{equation}\label{DLN}
				\frac{1}{K(\omega)}\|\ell\|_{BV^*}\le\|\ell\|_\omega^*\le \|\ell\|_{BV^*},
			\end{equation}
			where $K\colon \Omega \to [1, +\infty)$ is as in~\eqref{ln1};
			\item  For $\mathbb P$-a.e $\omega\in\Omega$ and  $\ell\in BV^*$,
			\begin{equation}\label{dan3}
				\|L_\omega^*\ell\|^*_\omega \le \|\ell\|_{\sigma\omega}^*;
			\end{equation}
			\item For $\mathbb P$-a.e $\omega\in\Omega$,  $\ell\in BV^*$ and  $n\in \N$,
			\begin{equation}\label{dan2}
				\|(L_\omega^n)^*\Pi^*(\sigma^n\omega)\ell\|^*_{\omega}\le e^{-\lambda' n}\|\ell\|_{\sigma^n\omega}^*,
			\end{equation}
			where $\lambda'$ is as in~\eqref{ln2}  and\footnote{We identify $\mu_\omega$ with the functional $\phi \mapsto \int_X \phi \, d\mu_\omega$ on $BV$.} $\Pi^*(\omega)\ell:=\ell-\ell(1)\mu_\omega$;
			\item we have that 
			\begin{equation}\label{dan4}
				\esssup_{\omega \in \Omega} \| \mu_\omega \|^*_\omega <+\infty;
			\end{equation}
			\item for $l\in BV^*$ and $\mathbb P$-a.e. $\omega \in \Omega$,
			\begin{equation}\label{dan5}
				|l(\mathds 1)| \le \|l\|_\omega^*.
			\end{equation}
		\end{enumerate}
	\end{lemma}
	
	\begin{proof} We follow closely the proof of~\cite[Proposition 33]{DDS}.
		For $\ell\in BV^*$ and $\mathbb P$-a.e. $\omega \in \Omega$,
		\[
		\|\ell\|^*_\omega:=\inf\{C>0: |\ell(\phi)|\le C\|\phi\|_\omega \ \text{for $\phi \in BV$}\},
		\]
		where $\| \cdot \|_\omega$, $\omega \in \Omega$ is the family of norms given by Lemma~\ref{PRO}.
		By~\eqref{ln1}, we have that 
		\[
		|\ell (\phi)| \le \| \ell \|_{BV^*} \cdot \| \phi \|_{BV} \le \| \ell \|_{BV^*} \cdot  \| \phi \|_\omega,
		\]
		for $\mathbb P$-a.e. $\omega \in \Omega$, $\ell \in BV^*$ and $\phi \in BV$. Hence, the second inequality in~\eqref{DLN} holds.  Moreover, 
		\[
		|\ell (\phi)| \le \| \ell \|_\omega^* \cdot \| \phi \|_\omega \le K(\omega)\| \ell \|_\omega^* \cdot \| \phi \|_{BV},
		\]
		for $\mathbb P$-a.e. $\omega \in \Omega$, $\ell \in BV^*$ and $\phi \in BV$, which yields the first inequality in~\eqref{DLN}.

		Furthermore, using~\eqref{ln4} we have that 
		\[
		|L_\omega^*\ell(\phi)| =|\ell (L_\omega \phi)| \le \| \ell\|_{\sigma\omega}^* \cdot \| L_\omega \phi \|_{\sigma \omega}\le \| \ell\|_{\sigma\omega}^* \cdot \| \phi \|_\omega,
		\]
		for $\mathbb P$-a.e. $\omega \in \Omega$, $\ell \in BV^*$ and $\phi \in BV$. Thus, \eqref{dan3} holds.
		
		On the other hand,  using~\eqref{ln2} we have that 
		\begin{align*}
			| (L_\omega^n)^*\Pi^*(\sigma^n\omega)\ell (\phi) | &=|\ell(L_\omega^n \tilde \Pi(\omega) \phi)|\\  
			&\le \| \ell \|_{\sigma^n \omega}^* \cdot \| L_\omega^n \tilde  \Pi(\omega) \phi \|_{\sigma^n \omega}\\ 
			&\le e^{-\lambda' n}\| \ell \|_{\sigma^n \omega}^* \cdot \| \phi \|_\omega, 
		\end{align*}
		for $\mathbb P$-a.e. $\omega \in \Omega$, $n\in \N$ and $\ell \in BV^*$. Therefore, \eqref{dan2} holds.
		
		Moreover, by~\eqref{ln3} we have that 
		\[
		|\mu_\omega(\phi)| =\bigg | \int_X \phi \, d\mu_\omega \bigg |  \le \| \phi \|_\omega, 
		\]
		for $\mathbb P$-a.e. $\omega \in \Omega$ and $\phi \in BV$. Hence, \eqref{dan4} holds. Finally, \eqref{dan5} follows readily from~\eqref{ln5}.
	\end{proof}

	\subsection{Proof of Theorem~\ref{TM}: perturbation results and consequences}

	Throughout this section, we consider a good cocycle $(\mathcal L_\omega)_{\omega \in \Omega}$, and an observable $\psi:\Omega\times X\to\R^d$, as in the statement of Theorem~\ref{TM}.
	By $|x|$ we  will denote the Euclidean norm of $x\in \mathbb C^d$. Moreover, we write $\psi_\omega$ instead of $\psi(\omega, \cdot)$.
	
	For $\theta \in \mathbb C^d$, $\omega \in \Omega$ and $\phi \in BV$, we (formally) set
	\[
	L_\omega^\theta \phi :=L_\omega (e^{\theta \cdot \psi_\omega} \phi).
	\]
	The proof of the following lemma is inspired by the proof of~\cite[Lemma 36]{DDS}.
	\begin{lemma}
		There exists $C'>0$ such that 
		\begin{equation}\label{twist}
			\| L_\omega^\theta \phi \|_{\sigma \omega} \le C' \| \phi \|_\omega \quad \text{for $\mathbb P$-a.e. $\omega \in \Omega$, $| \theta | \le 1$ and $\phi \in BV$,}
		\end{equation}
		where $\| \cdot \|_\omega$, $\omega \in \Omega$ is the family of norms given by Lemma~\ref{PRO}.
	\end{lemma}
	
	\begin{proof}
		Since $K(\omega) \ge 1$, it follows from~\eqref{eq:condobservable} that 
		\begin{equation}\label{EQ}
			\esssup_{\omega \in \Omega}   \|\psi_\omega^i\|_{BV} <+\infty, 
		\end{equation}
		for $1\le i \le d$.
		Take $\theta \in \mathbb C^d$ such that $|\theta| \le 1$.
		By~\eqref{ln1} and~\eqref{ln4}, we have that 
		\begin{equation}\label{AA}
			\begin{split}
				\| L_\omega^\theta \phi-L_\omega \phi \|_{\sigma \omega} &=  \| L_\omega ( (e^{\theta \cdot \psi_\omega}-1) \phi ) \|_{\sigma \omega} \\
				&\le  \|(e^{\theta \cdot \psi_\omega}-1) \phi )\|_\omega \\
				&\le K(\omega)  \|(e^{\theta \cdot \psi_\omega}-1) \phi )\|_{BV} \\
				&\le C_{var}K(\omega) \| e^{\theta \cdot \psi_\omega}-1\|_{BV} \cdot \| \phi \|_{BV} \\
				&\le C_{var} K(\omega) \| e^{\theta \cdot \psi_\omega}-1\|_{BV} \cdot \| \phi \|_\omega,
			\end{split}
		\end{equation}
		for $\mathbb P$-a.e. $\omega \in \Omega$ and $\phi \in BV$. 
		
		On the other hand, we have that 
		\begin{equation}\label{r1}
			\begin{split}
				\|  e^{\theta \cdot \psi_\omega}-1\|_{BV} &=\left\| \prod_{i=1}^d e^{\theta_i \psi_\omega^i} -1 \right\|_{BV} \\
				&\le C_{var}^d \sum_{i=1}^d \prod_{j=1}^{i-1} \|  e^{\theta_j \psi_\omega^j} \|_{BV} \cdot \| e^{\theta_i \psi_\omega^i }-1\|_{BV}
			\end{split}
		\end{equation}
		Moreover, for $1\le i \le d$ we have (see the proof of~\cite[Lemma 36]{DDS}) that 
		\begin{equation}\label{r2}
			\| e^{\theta_i \psi_\omega^i }-1\|_{BV} \le (1+C_{var})e^{\| \psi_\omega^i \|_\infty} \| \psi_\omega^i \|_{BV} 
		\end{equation}
		and consequently 
		\begin{equation}\label{r3}
			\| e^{\theta_i \psi_\omega^i }\|_{BV} \le 1+(1+C_{var})e^{\| \psi_\omega^i \|_\infty} \| \psi_\omega^i \|_{BV}.
		\end{equation}
		It follows from~\eqref{EQ}, \eqref{r1}, \eqref{r2} and~\eqref{r3}  that  there exists a constant $D>0$ such that 
		\begin{equation}\label{A}
			\|  e^{\theta \cdot \psi_\omega}-1\|_{BV} \le D \sum_{i=1}^d \| \psi_\omega^i \|_{BV}, \quad \text{for $\mathbb P$-a.e. $\omega \in \Omega$.}
		\end{equation}
		By~\eqref{eq:condobservable}, \eqref{AA} and~\eqref{A}, we find that there exists another constant $D'>0$ such that 
		\[
		\| L_\omega^\theta \phi-L_\omega \phi \|_{\sigma \omega} \le D'\| \phi \|_\omega, \quad \text{for $\mathbb P$-a.e. $\omega \in \Omega$ and $\phi \in BV$.}
		\]
		Finally, we observe that~\eqref{twist} follows readily from~\eqref{ln4} together with the above estimate. The proof of the lemma is completed. 
	\end{proof}

	Let $\mathcal S$ denote the space consisting of all measurable $\mathcal V \colon \Omega \times X \to \R$ such that $\mathcal V_\omega:=\mathcal V(\omega, \cdot) \in BV$ for $\mathbb P$-a.e. $\omega \in \Omega$ and 
	\[
	\| \mathcal V\|_{\mathcal S}:= \esssup_{\omega \in \Omega} \| \mathcal V_\omega\|_\omega <+\infty,
	\]
	where $\| \cdot \|_\omega$, $\omega \in \Omega$ is the family of norms given by Lemma~\ref{PRO}.
	Then, $(\mathcal S, \| \cdot \|_{\mathcal S})$ is a Banach space. 
	
	Furthermore, let $\mathcal S_0$ denote the set of all $\mathcal V\in \mathcal S$ such that 
	\[
	\int_X \mathcal V_\omega \, d\mu_\omega=0, \quad \text{for $\mathbb P$-a.e. $\omega \in \Omega$.}
	\]
	Using~\eqref{ln3}, it is easy to verify that $\mathcal S_0$ is a closed subspace of $\mathcal S$.
	
	For $(\theta, \mathcal W)\in \mathbb C^d \times \mathcal S_0$, we formally define
	\[
	F(\theta, \mathcal W)(\omega, \cdot)=\frac{L_{\sigma^{-1} \omega}^\theta(\mathds 1+\mathcal W_{\sigma^{-1} \omega})}{\int_X L_{\sigma^{-1} \omega}^\theta (\mathds 1+\mathcal W_{\sigma^{-1} \omega})\, d\mu_\omega}-\mathds 1-\mathcal W_\omega, \quad \omega \in \Omega.
	\]
	By arguing as in the proof of~\cite[Lemma 41]{DDS}, one can establish the following result.
	\begin{lemma}\label{L1}
		There exists a neighborhood $\mathcal U$ of $(0,0) \in \mathbb C^d\times \mathcal S_0$ such that $F\colon \mathcal U \to \mathcal S_0$ is well-defined and analytic. 
		Furthermore, its differential w.r.t $\mathcal W$ at $(0,0)$, $D_2F(0,0)\colon \mathcal S_0\to\mathcal S_0$ is invertible. 
	\end{lemma}

	The following result follows from Lemma~\ref{L1} and the implicit function theorem (exactly as in the proof of~\cite[Theorem 42]{DDS}).
	\begin{lemma}\label{L}
		There exists a neighborhood $U$ of $0\in\mathbb C^d$, such that for any $\theta\in U$, there exist $v^\theta\in\mathcal S$, $\lambda^\theta\in L^\infty(\Omega)$, satisfying:
		\begin{enumerate}
			\item The maps $U\ni \theta \mapsto v^\theta\in\mathcal S$ and $U\ni \theta \mapsto \lambda^\theta\in L^\infty(\Omega)$ are analytic.
			\item For any $\theta\in U$ and $\mathbb P$-a.e. $\omega\in\Omega$, $v_\omega^\theta$, $\lambda^\theta_\omega$ satisfy:
			\[
			\begin{split}
				L_\omega^\theta v_\omega^\theta&=\lambda_\omega^\theta v_{\sigma\omega}^\theta,\\
				\lambda_\omega^\theta&=\int_X L_\omega^\theta v_\omega^\theta \, d\mu_{\sigma \omega},\\
				1&=\int_X v_\omega^\theta \, d\mu_\omega.
			\end{split}
			\]
		\end{enumerate}
	\end{lemma}
	\begin{remark}
		As noted, Lemma~\ref{L} is close in spirit to~\cite[Theorem 42]{DDS}. However, there are some important differences. Indeed, in~\cite[Theorem 42]{DDS} we considered the case when $d=1$ and our perturbation result was stated for our original cocycle of transfer operators
		$(\mathcal L_\omega)_{\omega \in \Omega}$, while here we deal with the associated cocycle of normalized transfer operators $(L_\omega)_{\omega \in \Omega}$. 
	\end{remark}
	
	Let $\mathcal N$ consist of all measurable $\Phi \colon \Omega \to BV^*$ such that
	\[
	\| \Phi \|_{\mathcal N}:=\esssup_{\omega \in \Omega} \| \Phi_\omega \|_\omega^* <+\infty,
	\]
	where $\| \cdot \|_\omega^*$, $\omega \in \Omega$ is the family of norms given by Lemma~\ref{LL}.
	By $\mathcal N_0$ we denote the subspace of $\mathcal N$ consisting of all $\Phi \in \mathcal N$ such that 
	\[
	\Phi_\omega(\mathds 1)=0, \quad \text{for $\mathbb P$-a.e. $\omega \in \Omega$.}
	\]
	Then, it follows easily from~\eqref{dan5} that $\mathcal N_0$ is a closed subspace of $\mathcal N$.
	
	For $(\theta, \mathcal W)\in \mathbb C^d \times \mathcal S_0$, we formally define
	\[
	F^*(\theta, \Phi)_\omega=\frac{(L_\omega^\theta)^*(\Phi_{\sigma \omega}+\mu_{\sigma \omega})}{(L_\omega^\theta)^*(\Phi_{\sigma \omega}+\mu_{\sigma \omega}) (\mathds 1)}-\Phi_\omega-\mu_\omega, \ \omega \in \Omega.
	\]
	One can show that $F^*$ is well-defined and analytic on  a neighborhood of $(0, 0) \in \mathbb C^d \times \mathcal S_0$. Moreover, by arguing as in the proof of~\cite[Proposition 44]{DDS} (see also~\cite[Remark 45]{DDS}), one has:
	\begin{lemma}
		There exists a neighborhood $U'$ of $0$ in $\mathbb C^d$ and an analytic map $U'\ni \theta \mapsto \phi^\theta \in \mathcal N$ such that 
		\[
		(L_\omega^\theta)^*\phi_{\sigma \omega}^\theta=\lambda_\omega^\theta \phi_\omega^\theta, \quad \text{for $\mathbb P$-a.e. $\omega \in \Omega$ and $\theta \in U'$.}
		\]
	\end{lemma}
	By arguing  as in the proof of~\cite[Lemma 59]{DDS}, one can also establish the following result.
	\begin{lemma}\label{20}
		There exists $r\in (0, 1)$ such that for $\theta \in \mathbb C^d$ sufficiently close to $0$, $\mathbb P$-a.e. $\omega \in \Omega$, $h\in BV$ and $n\in \N$, 
		\[
		\| L_\omega^{\theta, n}(h-\phi_\omega^\theta(h)v_\omega^\theta) \|_{\sigma^n \omega} \le r^n \| h \|_\omega.
		\]
	\end{lemma}
	The following auxiliary result plays an important role in the proof of Theorem~\ref{TM}.
	\begin{lemma}\label{new}
		There exist constants $K, \rho>0$ such that 
		\[
		\| L_\omega^{it, n} h\|_{\sigma^n \omega} \le K\| h\|_\omega, 
		\]
		for $t\in \R^d$, $|t| \le \rho$, $\mathbb P$-a.e. $\omega \in \Omega$, $h\in BV$ and $n\in \N$.
	\end{lemma}
	
	\begin{proof}
		
		We have that
		\[
		\begin{split}
			\int_X e^{it \cdot S_n \psi (\omega, \cdot) }\,  d\mu_\omega &= \int_X L^{it, n}_{\omega} \mathds 1 \, d\mu_{\sigma^n \omega} \\
			&=\int_X L_{\omega}^{it, n}(\mathds 1-v^{it}_{\omega}) \, d\mu_{\sigma^n \omega} + \int_X L_{\omega}^{it, n}v^{it}_{\omega} \, d\mu_{\sigma^n \omega}\\
			&= \int_X L_{\omega}^{it, n}(\mathds 1-v^{it}_{\omega}) \, d\mu_{\sigma^n \omega} +\prod_{j=0}^{n-1}\lambda^{it}_{\sigma^j\omega}\int_X v^{it}_{\omega} \, d\mu_\omega\\
			&=  \int_X L_{\omega}^{it, n}(\mathds 1-v^{it}_{\omega}) \, d\mu_{\sigma^n \omega} +\prod_{j=0}^{n-1}\lambda^{it}_{\sigma^j\omega},
		\end{split}
		\]
		where $S_n \psi$ is given by~\eqref{RBS}. Thus,
		\[
		\begin{split}
			\bigg |\prod_{j=0}^{n-1}\lambda^{it}_{\sigma^j\omega}\bigg | &\le \bigg |\int_X L_{\omega}^{it, n}(\mathds 1-v^{it}_{\omega}) \, d\mu_{\sigma^n \omega} \bigg | +\bigg |\int_X e^{it \cdot S_n \psi (\omega, \cdot) }\,  d\mu_\omega \bigg | \\
			&\le 1+\bigg | \int_X e^{it\cdot S_n \psi (\omega, \cdot)}(\mathds 1-v^{it}_{\omega}) \, d\mu_{ \omega} \bigg |  \\
			&\le 1+\|\mathds 1-v^{it}_\omega\|_{L^1(\mu_\omega)} \\
			&\le 1+ C_{\var}\|\mathds 1-v^{it}_\omega\|_{BV} \\
			&\le 1+C_{\var} \|\mathds 1-v^{it}_\omega\|_{\omega}\\
			&\le 1+C_{\var} \|\mathds 1-v^{it} \|_{\mathcal S},
		\end{split}
		\]
		from which it follows that there exists $\rho>0$ such that 
		\begin{equation}\label{t}
			\bigg |\prod_{j=0}^{n-1}\lambda^{it}_{\sigma^j\omega}\bigg | \le 2, \quad \text{for $\mathbb P$-a.e. $\omega \in \Omega$, $n\in \N$ and $t\in \R^d$, $|t| \le \rho$.}
		\end{equation}
		
		On the other hand, we have that 
		\[
		\begin{split}
			L_\omega^{it, n}h &= L_\omega^{it, n}(h-\phi_\omega^{it}(h)v_\omega^{it})+\phi_\omega^{it}(h)L_\omega^{it, n}v_\omega^{it} \\
			&=L_\omega^{it, n}(h-\phi_\omega^{it}(h)v_\omega^{it})+ \phi_\omega^{it}(h)\bigg ( \prod_{j=0}^{n-1} \lambda_{\sigma^j \omega}^{it} \bigg )v_{\sigma^n \omega}^{it},
		\end{split}
		\]
		and thus
		\begin{equation}\label{t1}
			\|L_\omega^{it, n}h \|_{\sigma^n \omega} \le \|L_\omega^{it, n}(h-\phi_\omega^{it}(h)v_\omega^{it}) \|_{\sigma^n \omega}+\|\phi^{it}\|_{\mathcal N}\cdot\bigg |\prod_{j=0}^{n-1}\lambda^{it}_{\sigma^j\omega}\bigg | \cdot \|v^{it}\|_{\mathcal S}\|h\|_\omega,
		\end{equation}
		for $\mathbb P$-a.e. $\omega \in \Omega$, $n\in \N$, $h\in BV$ and $t\in \R^d$ sufficiently close to $0$. The desired conclusion follows directly from~\eqref{t} and~\eqref{t1} together with Lemma~\ref{20} and the continuity of $t\to \|\phi^{it}\|_{\mathcal N}$.
	\end{proof}
	
	\subsection{Completing the proof of Theorem~\ref{TM}}
	In order to establish the first assertion of Theorem~\ref{TM}, we will rely on  the arguments in~\cite{DFGTV1, DH3, DDS}.
	By $\mathbb E_\omega (\phi)$ we will denote $\int_X \phi\, d\mu_\omega$. In addition, let $\tau$ and $\mu$ be as in~\eqref{SP} and~\eqref{mu}, respectively.

	Firstly,  as in  \cite[Proposition 4.17]{DH3}, by replacing $\psi$ with $\psi\cdot v$ for an arbitrary unit vector $v$, it is enough to consider scalar valued functions $\psi$.
	In the scalar case, by using~\eqref{EQ} and arguing as in the proof of~\cite[Lemma 12]{DFGTV1}, we find that
	\begin{small}
		\[
		\mathbb E_\omega \bigg{(}\sum_{k=0}^{n-1}  \psi_{\sigma^k \omega} \circ T_\omega^k \bigg{)}^2=\sum_{k=0}^{n-1} \mathbb E_\omega ( (\psi_{\sigma^k \omega})^2\circ T_\omega^k )
		+2\sum_{i=0}^{n-1}\sum_{j=i+1}^{n-1}\mathbb E_{\sigma^i \omega}( \psi_{\sigma^i \omega}( \psi_{\sigma^j \omega}\circ T_{\sigma^i \omega}^{j-i}))
		\]
	\end{small}
	and 
	\[
	\lim_{n\to \infty} \frac 1 n  \sum_{k=0}^{n-1}  \mathbb E_\omega (( \psi_{\sigma^k \omega})^2 \circ T_\omega^k )=\int_{\Omega \times X}  \psi(\omega, x)^2\, d\mu (\omega, x),
	\]
	for $\mathbb P$-a.e. $\omega \in \Omega$. Set 
	\[
	\Psi(\omega)=\sum_{n=1}^\infty \int_X   \psi(\omega, x) \psi(\tau^n (\omega, x))\, d\mu_\omega(x)=\sum_{n=1}^\infty \int_X L_\omega^n(\psi_\omega) \psi_{\sigma^n \omega} \, d\mu_{\sigma^n \omega}.
	\]
	By~\eqref{ln1}, \eqref{ln2}  and~\eqref{ln3}, we have that 
	\begin{equation}\label{eq:same}
		\begin{split}
			\left| \int_X L_\omega^n(\psi_\omega)\psi_{\sigma^n\omega} d\mu_{\sigma^n\omega}\right|  &\le \| L_\omega^n(\psi_\omega)\cdot\psi_{\sigma^n \omega}\|_{\sigma^n\omega} \\
                        &\le  K(\sigma^n \omega) \|L_\omega^n(\psi_\omega)\cdot\psi_{\sigma^n \omega} \|_{BV} \\
			&\le C_{var}K(\sigma^n\omega) \| L_\omega^n(\psi_\omega) \|_{BV} \cdot \|\psi_{\sigma^n \omega}\|_{BV} \\
                       &\le C_{var}K(\sigma^n\omega) \| L_\omega^n(\psi_\omega) \|_{\sigma^n \omega} \cdot \|\psi_{\sigma^n \omega}\|_{BV} \\
			&\le C_{var}K(\sigma^n\omega)e^{-\lambda' n} \| \psi_\omega \|_\omega \cdot \|\psi_{\sigma^n \omega}\|_{BV}\\
			&\le C_{var}e^{-\lambda' n}K(\omega)\|\psi_\omega\|_{BV}K(\sigma^n\omega)\|\psi_{\sigma^n\omega}\|_{BV}, 
		\end{split}
	\end{equation}
for $\mathbb P$-a.e. $\omega \in \Omega$ and $n\in \N$.
	Therefore by \eqref{eq:condobservable} we get  that $\esssup_{\omega \in \Omega} |\Psi(\omega)|<+\infty$. Thus, $\Psi \in L^1(\Omega, \mathcal F, \mathbb P)$. Hence, it follows from Birkhoff's ergodic theorem that, for $\mathbb P$  a.e. $\omega \in \Omega$,
	\begin{align}\label{DD}
		\notag\lim_{n\to \infty}\frac{1}{n} \sum_{i=0}^{n-1}\Psi(\sigma^i \omega)&=\int_\Omega \Psi(\omega)\, d\mathbb P(\omega)\\
		&=\sum_{n=1}^\infty \int_{ \Omega \times X}\psi (\omega, x)\psi (\tau^n(\omega, x))\, d\mu (\omega, x).
	\end{align}
	Moreover,  we have that 
	\[
	\begin{split}
		&  \bigg{\lvert} \sum_{i=0}^{n-1}\sum_{j=i+1}^{n-1}\mathbb E_{\sigma^i \omega}( \psi_{\sigma^i \omega}( \psi_{\sigma^j \omega}\circ T_{\sigma^i \omega}^{j-i}))-\sum_{i=0}^{n-1}\Psi(\sigma^i \omega) \bigg{\rvert} \\
		& \le \sum_{i=0}^{n-1}\sum_{k=n-i}^\infty  \bigg{\lvert}\int_X L_{\sigma^i \omega}^k( \psi_{\sigma^i \omega} )  \psi_{\sigma^{k+i} \omega}\, d\mu_{\sigma^{k+i} \omega} \bigg{\rvert}  \\
		&\le C_{\var} \left(\esssup_{\omega \in \Omega} K(\omega)\|\psi_\omega\|_{BV}\right)^2 \sum_{i=0}^{n-1}\sum_{k=n-i}^\infty e^{-\lambda' k}.
	\end{split} 
	\]
	We thus derive that
	\begin{equation}\label{cv}
		\lim_{n\to \infty} \frac 1 n \bigg{(}\sum_{i=0}^{n-1}\sum_{j=i+1}^{n-1}\mathbb E_{\sigma^i \omega}(\psi_{\sigma^i \omega}( \psi_{\sigma^j \omega}\circ T_{\sigma^i \omega}^{j-i}))-\sum_{i=0}^{n-1}\Psi(\sigma^i \omega)\bigg{)}=0,
	\end{equation}
	It follows from~\eqref{DD} and~\eqref{cv} that
	\begin{small}
		\[
		\lim_{n \to \infty} \frac 1 n \sum_{i=0}^{n-1}\sum_{j=i+1}^{n-1}\mathbb E_{\sigma^i \omega}(\psi_{\sigma^i \omega}( \psi_{\sigma^j \omega}\circ T_{\sigma^i \omega}^{j-i}))=\sum_{n=1}^\infty \int_{ \Omega \times X} \psi(\omega, x)
		\psi(\tau^n(\omega, x))\, d\mu (\omega, x),
		\]
	\end{small}
	for $\mathbb P$-a.e. $\omega \in \Omega$. Hence, we conclude that 
	\[
	\begin{split}
		\lim_{n\to \infty} \frac 1 n \mathbb E_\omega \bigg{(}\sum_{k=0}^{n-1}  \psi_{\sigma^k \omega} \circ T_\omega^k \bigg{)}^2 &=\int_{\Omega \times X}  \psi(\omega, x)^2\, d\mu (\omega, x) \\
		&\phantom{=}+2\sum_{n=1}^\infty \int_{ \Omega \times X} \psi(\omega, x) 
		\psi(\tau^n(\omega, x))\, d\mu (\omega, x) \\
		&=:\Sigma^2 \ge 0,
	\end{split}
	\]
	for $\mathbb P$-a.e. $\omega \in \Omega$. In order to show that $\Sigma^2=0$ if and only if $\psi$ is a coboundary with $r\in L^2(\Omega\times X,\mu)$ (see~\eqref{Cob}), one can argue as in~\cite{DFGTV1}.
	The proof that we can get \eqref{Cob} with a function $r$ so that $\esssup_{\omega\in\omega}\|r(\omega,\cdot)\|_{BV}<\infty$ is postponed to the next section, see Lemma \ref{BVcob} (note that it is enough to consider real valued functions $\psi$).

	We now establish the second assertion of Theorem~\ref{TM}. In the following lemma, we verify~\cite[condition (2.1)]{DH3}.

	\begin{lemma}
		There exist constants $C,c,\rho>0$ such that for $\mathbb P$-a.e. $\omega \in \Omega$,  any $n,m>0$, $b_1<b_2< \ldots <b_{n+m+1}$, $k>0$ and $t_1,\ldots ,t_{n+m}\in\mathbb R^d$ with $|t_j|\leq \rho$, we have that
		\begin{equation}\label{EQQ}
			\begin{split}
				\Big|\mathbb E_\omega&\big(e^{i\sum_{j=1}^nt_j \cdot (\sum_{\ell=b_j}^{b_{j+1}-1}A_\ell)+i\sum_{j=n+1}^{n+m}t_j \cdot (\sum_{\ell=b_j+k}^{b_{j+1}+k-1}A_\ell)}\big)\\
				&-\mathbb E_\omega\big(e^{i\sum_{j=1}^nt_j \cdot (\sum_{\ell=b_j}^{b_{j+1}-1}A_\ell)}\big)\cdot\mathbb E_\omega\big(e^{i\sum_{j=n+1}^{n+m}t_j \cdot (\sum_{\ell=b_j+k}^{b_{j+1}+k-1}A_\ell)}\big)\Big|\\
				&\leq C^{n+m} e^{-ck},
			\end{split}
		\end{equation}
		where \[A_\ell:=\psi_{\sigma^\ell \omega} \circ T_\omega^\ell, \quad \ell\in \N. \]
	\end{lemma}
	
	\begin{proof}
		Set
		\[
		Q_w h=\bigg (\int_X h\, d\mu_\omega \bigg )\mathds 1, \quad \text{for $\omega \in \Omega$ and $h\in BV$.}
		\]
		We have that 
		\[
		\begin{split}
			& \mathbb E_\omega\big(e^{i\sum_{j=1}^nt_j \cdot (\sum_{\ell=b_j}^{b_{j+1}-1}A_\ell)+i\sum_{j=n+1}^{n+m}t_j \cdot (\sum_{\ell=b_j+k}^{b_{j+1}+k-1}A_\ell)}\big) \\
			&= \mathbb E_ {\sigma^{b_{n+m+1}+k} \omega} \bigg (\prod_{j=n+1}^{n+m}L_{\sigma^{b_{j}+k}\omega}^{it_{j}, b_{j+1}-b_{j}}                       L_{\sigma^{b_{n+1}} \omega}^{k}\prod_{j=1}^{n}L_{\sigma^{b_j}\omega}^{it_j, b_{j+1}-b_j}(\mathds 1) \bigg ) \\
			&= \mathbb E_ {\sigma^{b_{n+m+1}+k} \omega} \bigg (\prod_{j=n+1}^{n+m}L_{\sigma^{b_{j}+k}\omega}^{it_{j}, b_{j+1}-b_{j}}(L_{\sigma^{b_{n+1}} \omega}^{k}-Q_{\sigma^{b_{n+1}} \omega})\prod_{j=1}^{n}L_{\sigma^{b_j}\omega}^{it_j, b_{j+1}-b_j}(\mathds 1) \bigg ) \\
			&\phantom{=}+\mathbb E_ {\sigma^{b_{n+m+1}+k} \omega} \bigg (\prod_{j=n+1}^{n+m}L_{\sigma^{b_{j}+k}\omega}^{it_{j}, b_{j+1}-b_{j}}Q_{\sigma^{b_{n+1}} \omega}\prod_{j=1}^{n}L_{\sigma^{b_j}\omega}^{it_j, b_{j+1}-b_j}(\mathds 1) \bigg ).
		\end{split}
		\]
		It follows from~\eqref{ln2}, \eqref{ln3}, \eqref{ln5}  and Lemma~\ref{new} that 
		\[
		\begin{split}
			&\bigg |\mathbb E_ {\sigma^{b_{n+m+1}+k} \omega} \bigg ( \prod_{j=n+1}^{n+m}L_{\sigma^{b_{j}+k}\omega}^{it_{j}, b_{j+1}-b_{j}}(L_{\sigma^{b_{n+1}} \omega}^{k}-Q_{\sigma^{b_{n+1}} \omega}) \prod_{j=1}^{n}L_{\sigma^{b_j}\omega}^{it_j, b_{j+1}-b_j}(\mathds 1)\bigg )  \bigg | \\
			&\le K^{n+m} e^{-\lambda'k} \esssup_{\omega \in \Omega}\|\mathds 1\|_\omega \\
			&=K^{n+m} e^{-\lambda'k}.
		\end{split}
		\]
		The conclusion of the lemma follows from an observation that 
		\[
		\begin{split}
			&\mathbb E_ {\sigma^{b_{n+m+1}+k} \omega} \bigg (\prod_{j=n+1}^{n+m}L_{\sigma^{b_{j}+k}\omega}^{it_{j}, b_{j+1}-b_{j}} Q_{\sigma^{b_{n+1}} \omega} \prod_{j=1}^{n}L_{\sigma^{b_j}\omega}^{it_j, b_{j+1}-b_j}(\mathds 1) \bigg )\\
			&\phantom{+}=\mathbb E_\omega\big(e^{i\sum_{j=1}^nt_j \cdot (\sum_{\ell=b_j}^{b_{j+1}-1}A_\ell)}\big)\cdot\mathbb E_\omega\big(e^{i\sum_{j=n+1}^{n+m}t_j \cdot (\sum_{\ell=b_j+k}^{b_{j+1}+k-1}A_\ell)}\big).
		\end{split}
		\]
	\end{proof}
	
	Next we verify~\cite[condition (2.5)]{DH3}.
	\begin{lemma}
		There exist constants $C_0>0$ and $r\in (0, 1)$ such that 
		\[
		|\Cov (A_n \cdot v, A_{n+k} \cdot v)| \le C_0 r^k,
		\]
		for $\mathbb P$-a.e. $\omega \in \Omega$, $n, k\in \N$ and $v\in \R^d$ such that $|v|= 1$, where $A_n=\psi_{\sigma^n\omega} \circ T_\omega^n$.
	\end{lemma}
	
	\begin{proof}
		We have that 
		\begin{equation}\label{W1}
			\begin{split}
				\Cov (A_n \cdot v, A_{n+k} \cdot v) &=\sum_{1\le i, j\le d} \int_X v_i v_j \psi_{\sigma^n \omega}^i \circ T_\omega^n \cdot  \psi_{\sigma^{n+k} \omega}^j \circ T_\omega^{n+k} \, d\mu_\omega  \\
				&=\sum_{1\le i, j\le d} v_i v_j\int_X  \psi_{\sigma^n \omega}^i \cdot \psi_{\sigma^{n+k} \omega}^j \circ T_{\sigma^n \omega}^{k} \, d\mu_{\sigma^n \omega} \\
				&=\sum_{1\le i, j\le d} v_i v_j\int_X L_{\sigma^n \omega}^k(\psi_{\sigma^n \omega}^i ) \psi_{\sigma^{n+k} \omega}^j\, d\mu_{\sigma^{n+k} \omega}.
			\end{split}
		\end{equation}
		The same computation as \eqref{eq:same} now yields
		\begin{equation}\label{W2}
			\begin{split}
				\bigg |\int_X L_{\sigma^n \omega}^k(\psi_{\sigma^n \omega}^i )& \psi_{\sigma^{n+k} \omega}^j\, d\mu_{\sigma^{n+k} \omega}\bigg |\\ 
				&\le  C_{var} e^{-\lambda' k} K(\sigma^n\omega)\| \psi_{\sigma^n \omega}^i\|_{BV}K(\sigma^{n+k}\omega)\| \psi_{\sigma^{n+k} \omega}^j \|_{BV},
			\end{split}
		\end{equation}
		which, together with \eqref{eq:condobservable} and \eqref{W1} imply the conclusion of the lemma.
	\end{proof}
	
	The conclusion of  Theorem~\ref{TM} now follows directly from the previous two lemmas by applying the abstract version of ASIP given in~\cite[Theorem 2.1]{DH3}.
	
	\begin{remark}\label{ej}
		We can now explain the reason for introducing adapted norms. We first observe that it follows from~\eqref{ln1} and Lemma~\ref{new} that 
		\begin{equation}\label{441}
			\| L_\omega^{it, n} \|_{BV} \le M(\omega) \quad \text{for $t\in \R^d$, $|t| \le \rho$, $n\in \N$ and $\mathbb P$-a.e. $\omega \in \Omega$,}
		\end{equation}
		where $M\colon \Omega \to (0, +\infty)$ is a tempered random variable.  By  relying solely  on~\eqref{est11} and~\eqref{441}, we observe that the L.H.S in~\eqref{EQQ} can be bounded by 
		\[
		\bigg (\prod_{j=n+1}^{n+m} M(\sigma^{b_{j}+k}\omega) \bigg ) \tilde D(\sigma^{b_{n+1}} \omega)e^{-\lambda' k} \bigg (\prod_{j=1}^{n}M(\sigma^{b_j}\omega) \bigg ).
		\]
		Note that even in the case when $M$ is a  constant, the above expression depends on $\omega$ and thus~\cite[Theorem 2.1]{DH3} is not directly applicable.  Our construction of adapted norms is precisely tailored to overcome this difficulty. 
		
	\end{remark}

	\section{A scalar-valued almost sure invariance principle}\label{S5}
	In this section we will improve the rates obtained in the previous section for a class of real valued observables. In order to achieve that, we need to impose an additional requirement. More precisely, we suppose that there exists a tempered random variable $N\colon \Omega \to (0, +\infty)$ such that 
	\begin{equation}\label{Nom}
		\| g\circ T_\omega \|_{BV} \le N(\omega) \| g\|_{BV}, \quad \text{for $\mathbb P$-a.e. $\omega \in \Omega$ and $g\in BV$.}
	\end{equation}
	\begin{remark}
		In the setting of Example~\ref{ex:good}, the above condition is satisfied whenever the number $N_\omega$ of monotonicity intervals is tempered.
	\end{remark}

	\begin{theorem}\label{TMA}
		Let $\mathcal L=(\mathcal L_\omega)_{\omega \in \Omega}$ be a good cocycle of transfer operators. Then, there exists a tempered random variable $\tilde K \colon \Omega \to [1, +\infty)$ with the property that for each measurable observable $\psi \colon \Omega \times X \to \R$ satisfying the following properties:
		
		\begin{equation}\label{bb}
			\esssup_{\omega \in \Omega} \bigg ( \tilde K(\omega) \| \psi_\omega \|_{BV} \bigg )<+\infty,
		\end{equation}
		and for $\mathbb P$-a.e. $\omega \in \Omega$,
		\[
		\int_X \psi_\omega \, d\mu_\omega=0,
		\]
		the following holds:
		\begin{enumerate}
			\item there exists $\Sigma^2\ge 0$ such that
			\[
			\Sigma^2=\lim_{n\to \infty}\frac 1 n \mathbb E_\omega \bigg{(}\sum_{k=0}^{n-1}\psi_{\sigma^k \omega}\circ T_\omega^{k} \bigg{)}^2, \quad \text{for $\mathbb P$-a.e. $\omega \in \Omega$;}
			\]
			Moreover, $\Sigma^2=0$ if and only if $\psi=r-r\circ\tau$ for some measurable $r:\Omega\times X\to\mathbb R$ so that $\esssup_{\omega\in\Omega}\|r(\omega,\cdot)\|_{BV}<\infty$.
			\item Assume $\Sigma^2 >0$.
			Then, for  $\mathbb P$-a.e. $\omega \in \Omega$ and  $\forall \epsilon>\frac 34,$  by enlarging the  probability space $(X, \mathcal B, \mu_\omega)$ if necessary, it is possible to find a sequence $(Z_k)_k$ of independent and centered (i.e. of zero mean) Gaussian random variables
			such that $\mu_\omega$ almost surely,
			\begin{equation*}
				\sup_{1\le k\le n}\left| \sum_{i=0}^{k-1} (\psi_{\sigma^i\omega} \circ T_\omega^i)-\sum_{i=1}^k Z_i \right|=O(n^{1/4}\log^{\epsilon}(n)).
			\end{equation*}
			Moreover, the difference between the $L^2$-norms of $\sum_{i=0}^{k-1} (\psi_{\sigma^i\omega} \circ T_\omega^i)$ and \\$\sum_{i=1}^k Z_i$ is bounded in $k$, and the variance of $Z_i$ equals $\int_X m_{\sigma^i\omega}^2d\mu_{\sigma^i\omega}$, with $m_\omega$ given by \eqref{m def}.
		\end{enumerate}
	\end{theorem}

	\begin{remark}
		We note that $\tilde K$ will be constructed so that  $K\le \tilde K$, where $K$ is as in the statement of Theorem~\ref{TM}. Consequently, the first part of Theorem~\ref{TMA} will be a consequence of Theorem~\ref{TM}.
	\end{remark}
	
	Let $K_1=\max \{ K, N, \tilde D \}$, where $K$ is as in the statement of Theorem~\ref{TM}, $\tilde D$ is as in the statement of Corollary~\ref{Cor} and $N$ is from~\eqref{Nom}. By taking into account Proposition~\ref{PA}, we can assume that 
	\begin{equation}\label{K 1 prop}
		K_1(\sigma^n \omega) \le K_1(\omega)e^{\delta |n|} \quad \text{for $\mathbb P$-a.e. $\omega \in \Omega$ and $n\in \Z$, }
	\end{equation}
	where $\delta \in (0, 2\lambda'/5)$ is fixed but arbitrary and $\lambda'>0$ is given by Corollary~\ref{Cor}. Finally, set $\tilde K=K_1^{7/2}$. Then, $\tilde K$ is  tempered.

	Set
	\begin{equation}\label{eq:DefChi}
		\chi_\omega:=\sum_{n=1}^\infty L_{\sigma^{-n} \omega}^n (\psi_{\sigma^{-n} \omega}), \quad \omega \in \Omega. 
	\end{equation}
	\begin{lemma}\label{LemChi}
		We have that 
		\begin{equation}\label{LC0}
		\esssup_{\omega \in \Omega}(K_1(\omega)^{5/2} \| \chi_\omega \|_{BV})<+\infty.
		\end{equation}
	\end{lemma}
	\begin{proof}
		By~\eqref{est11} and~\eqref{K 1 prop}, it follows that 
		\[
		\begin{split} 
			\| \chi_\omega \|_{BV} &\le \sum_{n=1}^\infty \tilde D(\sigma^{-n} \omega)e^{-\lambda' n} \|\psi_{\sigma^{-n} \omega}\|_{BV} \\
			&\le \sum_{n=1}^\infty K_1(\sigma^{-n} \omega)e^{-\lambda' n}\|\psi_{\sigma^{-n} \omega}\|_{BV} \\
			&\le \esssup_{\omega \in \Omega} \bigg ( \tilde K(\omega) \| \psi_\omega \|_{BV} \bigg ) \sum_{n=1}^\infty \frac{1}{ K_1(\sigma^{-n} \omega)^{5/2}}e^{-\lambda' n} \\
			&\le \esssup_{\omega \in \Omega} \left( \tilde K(\omega) \| \psi_\omega \|_{BV} \right)K_1(\omega)^{-5/2}\sum_{n=1}^\infty e^{-(\lambda'-5\delta/2)n},
		\end{split}
		\]
		for $\mathbb P$-a.e. $\omega \in \Omega$. Since $\delta \in (0, 2\lambda'/5)$, we conclude that the statement of the lemma holds. 
	\end{proof}
	
	Next, set
	\begin{equation}\label{m def}
		m_\omega=\psi_\omega+\chi_\omega-\chi_{\sigma \omega} \circ T_\omega, \quad \omega \in \Omega.
	\end{equation}
	\begin{lemma}\label{M}
		We have
		\[
		\esssup_{\omega \in \Omega}K_1(\omega)^{3/2}\| m_\omega \|_{BV}<+\infty.
		\]
	\end{lemma}
	\begin{proof}
		We have that 
		\[
		\begin{split}
			\| m_\omega \|_{BV} & \le \|\psi_\omega \|_{BV} + \|\chi_\omega \|_{BV}+N(\omega) \|\chi_{\sigma \omega}\|_{BV} \\
			&\le \|\psi_\omega \|_{BV} + \|\chi_\omega \|_{BV}+K_1(\omega)\|\chi_{\sigma \omega}\|_{BV},
		\end{split}
		\]
		for $\mathbb P$-a.e. $\omega \in \Omega$, which together with $K_1(\omega) \ge 1$, Lemma \ref{LemChi} and \eqref{K 1 prop} easily implies the lemma.
	\end{proof} 
	
	Before proceeding with the proof of the ASIP, we complete the proof about the existence of a coboundary representation with a $BV$ function:
	
	\begin{lemma}\label{BVcob}
		Suppose that there exists a measurable map $c\colon \Omega \times X\to \mathbb R$ such that
		\begin{equation}\label{kck}
			\psi=c\circ \tau -c \quad \text{and} \quad \int_{\Omega \times X} |c(\omega,x)|^2d\mu(\omega,x)<\infty.
		\end{equation}
		Then for $\mathbb P$-a.e $\omega\in\Omega$, $c_\omega:=c(\omega,\cdot)\in BV$ and $\esssup_{\omega\in\Omega}\|c(\omega,\cdot)\|_{BV}<\infty$. 
	\end{lemma}
	\begin{proof}
		First, notice that the function $c$ in \eqref{kck} satisfies
		$c_\omega=c(\omega,\cdot)\in L^2(\mu_\omega)$ for $\mathbb P$-a.e. $\omega \in \Omega$.
		Next, by Lemma \ref{LemChi} we have
		\begin{equation}\label{yyy}
			\esssup_{\omega \in \Omega}\lVert \chi_\omega\rVert_{BV}<\infty.
		\end{equation}
		The rest of the proof will only rely on \eqref{yyy} (and not the stronger condition \eqref{LC0}) and thus remains valid in the circumstances of Theorem \ref{TM}, since the arguments in the proof of Lemma \ref{LemChi} give \eqref{yyy} when $K(\omega)\|\psi_\omega\|_{BV}$ is bounded.
		A straightforward computation yields that
		\[
		\chi_\omega-L_{\sigma^{-1}\omega}\chi_{\sigma^{-1}\omega}=L_{\sigma^{-1}\omega}\psi_{\sigma^{-1}\omega}, \quad \text{for $\mathbb P$-a.e. $\omega \in \Omega$.}
		\]
		On the other hand, \eqref{kck} together with the fact that $L_{\sigma^{-1}\omega}(c_\omega\circ T_{\sigma^{-1}\omega})=c_\omega$ (see ~\cite[Lemma 7]{DFGTV1}) imply that 
		\[
		c_\omega-L_{\sigma^{-1}\omega}c_{\sigma^{-1}\omega}=L_{\sigma^{-1} \omega}\psi_{\sigma^{-1}\omega}, \quad \text{for $\mathbb P$-a.e. $\omega \in \Omega$.}
		\]
		Setting $d_\omega:=c_\omega-\chi_\omega$, it follows from the last two identities that
		\begin{equation}\label{I}
			d_{\sigma \omega}=L_\omega d_\omega \quad \text{for $\mathbb P$-a.e. $\omega \in \Omega$,}
		\end{equation}
		Since $\sigma$ is ergodic, we conclude that the integral $d_0:=\int_X d_\omega d\mu_\omega$ does not depend on $\omega$. Moreover, since $|L_\omega d_\omega|\leq L_\omega|d_\omega|$ we have that the norm $\|d_\omega\|_{L^1(\mu_\omega)}$ does not depend on $\omega$.
		Next, by iterating~\eqref{I} we obtain that 
		\begin{equation}\label{d rel}
			d_{\sigma^n\omega}=L_{\omega}^nd_{\omega}, \quad \text{for $\mathbb P$-a.e. $\omega \in \Omega$ and $n\in \N$.}
		\end{equation}
		Observe now that for all $g_1,g_2\in L^1(\mu_{\omega})$ we have
		\[
		\|L_{\omega}^ng_1-L_{\omega}^ng_2\|_{L^1(\mu_{\sigma^n\omega})}\leq \int_X L_{\omega}^n(|g_1-g_2|)d\mu_{\sigma^n\omega}=\|g_1-g_2\|_{L^1(\mu_\omega)}. 
		\]
		Now, since functions with bounded variation are dense in $L^1(\mu_\omega)$, by approximating $d_\omega$ by a function with bounded variation and then using \eqref{est11} we obtain that for $\mathbb P$ a.e. $\omega \in \Omega$, we have
		\begin{equation}\label{d rel2}
			\lim_{n\to\infty}\left\|L_{\omega}^nd_\omega-\left(\int_X d_\omega d\mu_\omega\right)\mathds 1\right\|_{L^1(\mu_{\sigma^n\omega})}=0.
		\end{equation}
		By~\eqref{d rel} and~\eqref{d rel2},  we get that
		\[
		\lim_{n\to\infty}\|d_{\sigma^n\omega}-d_0\|_{L^1(\mu_{\sigma^n\omega})}=0, \quad \text{for $\mathbb P$-a.e. $\omega \in \Omega$.}
		\]
		Since $\sigma$ is ergodic we conclude that $d_\omega=d_0$ in $L^1(\mu_\omega)$ for $\mathbb P$ a.e. $\omega \in \Omega$. Indeed,
		\begin{align*}
			\int_{\Omega\times X}|d(\omega,x)-d_0|d\mu(\omega,x)&=\int_{\Omega}\left(\int_X|d\circ\tau^n(\omega,x)-d_0|d\mu_\omega(x)\right) d\mathbb P(\omega)\\
			&=\int_{\Omega}\left(\int_X|d_{\sigma^n\omega}-d_0|d\mu_{\sigma^n\omega}(x)\right) d\mathbb P(\omega)\\
			&=\int_{\Omega}\|d_{\sigma^n\omega}-d_0\|_{L^1(\mu_{\sigma^n\omega})} d\mathbb P(\omega),
		\end{align*}
		where $d(\omega,x):=d_\omega(x)$.
		Next, observe that the random variables $G_n(\omega)=\|d_{\sigma^n\omega}-d_0\|_{L^1(\mu_{\sigma^n\omega})}$ are uniformly bounded since $\| d_\omega \|_{L^1(\mu_\omega)}$ does not depend on $\omega$. By using the dominated convergence theorem we conclude that the above right hand side converges to $0$.
		Therefore, we get that for $\mathbb P$-a.e. $\omega \in \Omega$,  $d_\omega=d_0$. Hence, 
		\[c_\omega=d_\omega+\chi_\omega=d_0+\chi_\omega\in BV,\]
		with $\esssup_{\omega\in\Omega}\|c_\omega\|_{BV}<+\infty$, for $c$ satisfying \eqref{kck}.
	\end{proof}	
	Going back to \eqref{M},  one has via a straightforward computation: 
	\[
	L_\omega (m_\omega)=0 \quad \text{for $\mathbb P$-a.e. $\omega \in \Omega$.}
	\]
	In particular, it follows (see~\cite[Proposition 2.]{DFGTV1}) that\footnote{$\mathbb E_\omega[\phi\rvert \mathcal G]$ denotes the conditional expectation of $\phi$ with respect to the $\sigma$-algebra $\mathcal G$ and $\mu_\omega$.}
	\[
	\mathbb E_\omega[m_{\sigma^n \omega}\circ T_\omega^{n}\rvert (T_\omega^{n+1})^{-1}(\mathcal B)]=L_{\sigma^n \omega}(m_{\sigma^n \omega})\circ T_\omega^{n+1}=0.
	\]
	In other words, $(m_{\sigma^n \omega}\circ T_\omega^{n})_{n\ge 0}$ is a so-called reversed martingale difference with respect to the sequence of $\sigma$-algebras $((T_{\omega}^n)^{-1}(\mathcal B))_{n\ge 0}$. In view of the above lemmas together with~\eqref{m def}, we have that 
	\begin{equation}\label{MarAp}
		\sup_n\,\esssup_{\omega \in \Omega}\left\| \sum_{k=0}^{n-1}\psi_{\sigma^k \omega}\circ T_\omega^{k}-\sum_{k=0}^{n-1}m_{\sigma^k \omega}\circ T_\omega^{k}\right\|_{\infty}<\infty.
	\end{equation}
	Next, we define a new observable $\hat{\psi}\colon \Omega \times X\to \mathbb R$ by 
	\[
	\hat{\psi}_\omega=L_\omega (m_\omega^2)\circ T_\omega -\int_X m_\omega^2\, d\mu_\omega, \quad \omega \in \Omega.
	\]
	Clearly, 
	\[
	\int_X\hat{\psi}_\omega \, d\mu_\omega=0, \quad \text{for $\mathbb P$-a.e. $\omega \in \Omega$.}
	\]
	Moreover,  we have  that
	\begin{equation}\label{hatpsi}
		\esssup_{\omega \in \Omega}K_1(\omega)\lVert \hat{\psi}_\omega \rVert_{BV} <\infty.
	\end{equation}
	Indeed,  it follows from~\eqref{est11}, \eqref{est22},  \eqref{Nom} and Lemma~\ref{M} that there exists $C>0$ such that 
	\[\|L_\omega (m_\omega^2)\circ T_\omega \|_{BV}\leq  2\tilde D(\omega)N(\omega)\|m_\omega^2\|_{BV}\leq 2 K_1(\omega)^2\|m_\omega^2\|_{BV}\leq CK_1(\omega)^{-1},
	\]
	for $\mathbb P$-a.e. $\omega \in \Omega$. The above estimate together with Lemma~\ref{M} implies that~\eqref{hatpsi} holds.
	We will need the following lemma.
	
	\begin{lemma}\label{1023j}
		For $\mathbb P$-a.e. $\omega \in \Omega$, we have that
		\[
		\sum_{k=0}^{n-1} \hat{\psi}_{\sigma^k \omega}\circ T_\omega^{k}=O(n^{1/2} \sqrt{\log \log n} ).
		\]
	\end{lemma}	
	
	\begin{proof}
		There are two possibilities. Either the variance associated to $\hat{\psi}$ is nonzero or zero. If it is nonzero, then it follows from Theorem~\ref{TM} together with \eqref{hatpsi} (recall also that $K_1\ge K$)  that the process $\big(\hat{\psi}_{\sigma^k \omega}\circ T_\omega^{k}\big)_k$ satisfies the ASIP. In particular the law of iterated logarithm holds true, which implies the desired conclusion.

		In the case the variance vanishes, as in the proof of Theorem \ref{TM} there exists a bounded measurable function $c\in L^2(\Omega\times X,\mu)$ so that $\hat{\psi}=c\circ \tau -c$, $\mu$ almost everywhere. Thus, 
		\[
		\sum_{k=0}^{n-1} \hat{\psi}_{\sigma^k \omega}\circ T_\omega^{k}=-c_{\omega}+c_{\sigma^{n}\omega}\circ T_\omega^{n}, \quad \text{for $\mathbb P$-a.e. $\omega \in \Omega$.}
		\]
		We remark that we may construct for $\hat\psi$, via \eqref{eq:DefChi}, an observable $\hat\chi$  satisfying  Lemma \ref{LemChi} and thus \eqref{yyy}, as we did for $\psi$. Hence, Lemma \ref{BVcob} applies, and we actually have $\esssup_{\omega\in\Omega}\|c(\omega,\cdot)\|_{BV}<\infty$. Using (V3) we then get that 
		\[
		\sum_{k=0}^{n-1} \hat{\psi}_{\sigma^k \omega}\circ T_\omega^{k}=O(1).
		\]
		Alternatively,  since $\mu$ is ergodic and $c^2$ integrable we have that 
		\[
		\lim_{n\to\infty}\frac1n\sum_{j=0}^{n-1}c_{\sigma^{j}\omega}^2 \circ T_\omega^{j}=\int_{\Omega \times X} c^2d\mu
		\]
		and so for $\mathbb P$-a.e $\omega\in\Omega$ and  for $\mu_\omega$-a.e. $x$ we have
		\[
		c_{\sigma^{n}\omega}^2 \circ T_\omega^{n}=o(n),
		\]
		which implies that 
		\[
		\sum_{k=0}^{n-1} \hat{\psi}_{\sigma^k \omega}\circ T_\omega^{k}=o(\sqrt n),
		\]
		giving us the announced result.
	\end{proof}
	Finally, we recall the following result.
	\begin{lemma}[\cite{CM}]\label{CMT}
		Let $(X_n)_n$ be a sequence of square integrable random variables adapted to a non-increasing filtration $(\mathcal G_n)_n$. Assume that $\mathbb E(X_n\rvert \mathcal G_{n+1})=0$ almost surely, that
		\begin{equation}\label{varn}
			v_n^2:=\sum_{k=1}^n \mathbb E(X_k^2) \underset{n\to \infty}{\longrightarrow} \infty
		\end{equation}
		and that $\sup_n \mathbb E(X_n^2) <\infty$. Moreover, let $(a_n)_n$ be a non-decreasing sequence of positive numbers
		such that $(a_n/v_n^2)_n$ is non-increasing, $(a_n/v_n)_n$ is non-decreasing and:
		\begin{equation}\label{CMT1}
			\sum_{k=1}^n (\mathbb E(X_k^2\rvert \mathcal G_{k+1})-\mathbb E(X_k^2))=o(a_n) \quad a.s.;
		\end{equation}
		\begin{equation}\label{CMT2}
			\sum_{n\ge 1} a_n^{-v}\mathbb E(\lvert X_n\rvert^{2v}) <\infty \quad \text{for some $1\le v\le 2$.}
		\end{equation}
		Then, up to enlarging our probability space, it is possible to find a sequence $(Z_k)_k$ of independent and centered Gaussian variables with $\mathbb E(X_k^2)=\mathbb E(Z_k^2)$ such that, almost surely:
		\[
		\sup_{1\le k\le n}\bigg{\lvert} \sum_{i=1}^k X_i-\sum_{i=1}^k Z_i \bigg{\rvert}=o\left((a_n (\lvert \log (v_n^2/a_n) \rvert +\log \log a_n))^{1/2}\right)
		\]
	\end{lemma}
	
	We are now in  a position to complete the proof of Theorem~\ref{TMA}.
	\begin{proof}[Proof of the Theorem~\ref{TMA}]
		Observe that 
		\[
		\begin{split}
			& \sum_{k=0}^{n-1} (\mathbb E_\omega [m_{\sigma^k \omega}^2\circ T_\omega^{k} \rvert (T_\omega^{k+1})^{-1}(\mathcal B)]-\mathbb E_\omega (m_{\sigma^k \omega}^2\circ T_\omega^{k}))\\
			&=\sum_{k=0}^{n-1} \bigg{(} L_{\sigma^k \omega}(m_{\sigma^k \omega}^2)\circ T_\omega^{k+1}-\int_X m_{\sigma^k \omega}^2\, d\mu_{\sigma^k \omega} \bigg{)} \\
			&=\sum_{k=0}^{n-1} \hat{\psi}_{\sigma^k \omega}\circ T_\omega^{k},
		\end{split}
		\]
		and thus it follows from Lemma~\ref{1023j} that for $\mathbb P$-a.e. $\omega \in \Omega$  we have that
		\[
		\sum_{k=0}^{n-1} \big(\mathbb E_\omega [m_{\sigma^k \omega}^2\circ T_\omega^{k} \rvert (T_\omega^{k+1})^{-1}(\mathcal B)]-\mathbb E_\omega (m_{\sigma^k \omega}^2\circ T_\omega^{k})\big)=O(n^{1/2}\sqrt{\log \log  n}).
		\]
		This, in particular, shows that~\eqref{CMT1} holds true with $X_n=m_{\sigma^n\omega}\circ T_\omega^{n}$ and 
		$a_n=n^{1/2}(\log  n)^\epsilon$, for any $\epsilon>\frac12$. We  next  show that ~\eqref{CMT2} holds true with $v=2$. Indeed, it follows from Lemma~\ref{M} that  there exists $C>0$ such that  
		\[
		\sum_{n\geq 1}a_n^{-2}\mathbb E_\omega[m_{\sigma^n\omega}^4\circ T_\omega^{n}]\leq C\sum_{n\geq 1}a_n^{-2}<\infty.
		\]
		By applying  Lemma~\ref{CMT} for $X_n=m_{\sigma^n\omega}\circ T_\omega^{n}$, $n\in \N$, using (\ref{MarAp}) and that $\Sigma^2>0$ (which insures that $v_n^2$ grows linearly fast in $n$), we complete the proof of the Theorem~\ref{TMA}.
	\end{proof}

	\begin{remark}
		Using \eqref{MarAp}, Lemma \ref{m def} and  the Chernoff bounding scheme we can obtain an
		exponential concentration inequality in the sense of \cite[Proposition 4.5]{DH}.
	\end{remark}

	\section{Appendix A: a counter example to the existence of the asymptotic variance}\label{App A}

In this section, we  present an explicit example of a good cocycle of transfer operators and an real-valued observable  such that the conclusion of Theorem~\ref{TM} fails. Our observable will satisfy all assumptions of Theorem~\ref{TM} except for~\eqref{eq:condobservable}.  We stress that the example is essentially taken from~\cite[Appendix A]{Buzzi}. 
	
	
	Consider $(\tilde\Omega,\mathcal B,\mathds Q, S)$ the full-shift over $\{1,2,\dots\}$, with probability vector $(Z,Z/2^{2+\delta},\dots,Z/n^{2+\delta},\dots)$, for some $0\le\delta\le 1$, $Z$ being the normalization constant.
	\\Let $h:\tilde\Omega\to\R$ be the (positive) observable defined by $h(\omega)=\omega_0$ if $\omega:=(\omega_n)_{n\in \mathbb Z} \in\tilde\Omega$. Note that
	\[\int_{\tilde\Omega}h~d\mathds Q=\sum_{i\ge 1}i\cdot\frac{Z}{i^{2+\delta}}=Z\sum_{i=1}^\infty\frac{1}{i^{1+\delta}}<+\infty,\]
	when $0<\delta\le 1$.
	\\Define $(\Omega,\mathcal F,\mathbb P,\sigma)$ to be the suspension over $S$ with roof function $h$, i.e. $\Omega:=\{(\omega,i)\in\tilde\Omega\times\N,~ 0\le i<h(\omega)\}$, $\sigma:\Omega\circlearrowleft$ is given by
	\[
	\sigma(\omega,i):=\left\{
	\begin{aligned}
	&(\omega,i+1)~&\text{if}~i<h(\omega)-1\\
	&(S\omega,0)~&\text{if}~i=h(\omega)-1
	\end{aligned}
	\right.
	\]
	and $\mathbb P(A):=\left(\int_{\tilde\Omega} h~ d\mathds Q\right)^{-1}\sum_{i\ge 0}\mathds Q\left(A\cap(\tilde\Omega\times\{i\})\right)$.
	\\We can now define our random system: take $T_0:[0,1]\circlearrowleft$ to be the doubling map and let $T_1(x)=\frac{1}{2}\left(E(2x)+\{4x\}\right)$, where $E(x)$ denotes the integer part of $x$ and $\{x\}$ its fractional part.  We consider the random interval map $T_{(\omega,i)}$, $(\omega,i)\in\Omega$ defined by
\[
T_{(\omega,i)}:=\begin{cases}
T_1 & \text{if $i<h(\omega)-1$} \\
T_0 & \text{if  $i=h(\omega)-1$.}
\end{cases}
\]
It is verified in~\cite{Buzzi} that the associated cocycle of transfer operators  $(\mathcal L_{(\omega, i)})_{(\omega, i) \in \Omega}$ is good in the sense of \cite[Definition 13]{DDS}.  Moreover, we have that $\mu_{(\omega, i)}=m$ for $(\omega, i) \in \Omega$, where $m$ denotes the Lebesgue measure on $[0, 1]$. Hence $\L_{(\omega,i)}\mathds 1=\mathds 1$, which implies that the log-integrability condition is trivially satisfied. Let
\[
n_c (\omega, i):=\min \{ k\in \N: T_{(\omega,i)}^k([0,1/2])=[0,1]\}, \quad (\omega, i) \in \Omega. 
\]
Then, we have (see~\cite[p.47]{Buzzi}) that $n_c(\omega, i)=h(\omega)-i$. Observe that $n_c$ is not integrable. Indeed, for each $N\in \N$ we have that 
\begin{align*}
	\mathbb P(n_c(\omega,i)=N)&=\left(\int_{\tilde\Omega} h~ d\mathds Q\right)^{-1}\sum_{i\ge 0}\mathds Q(h(\omega)-i=N)\\
	&=\left(\int_{\tilde\Omega} h~ d\mathds Q\right)^{-1}\sum_{i\ge N}\mathds Q(\omega_0=i)\\
	&=\left(\int_{\tilde\Omega} h~ d\mathds Q\right)^{-1}\sum_{i\ge N}\frac{Z}{i^{2+\delta}}\sim\frac{C}{N^{1+\delta}},
	\end{align*}
	for some constant $C>0$, which easily implies that $n_c$ is not integrable. 
	We can finally introduce the observables of interest: consider $\phi=2\cdot\mathds 1_{[0,1/2]}$, and $\psi=\phi-\int_{[0, 1]}\phi \, dm$. We have  (see~\cite[p.47]{Buzzi}) that 
	\[
	\int_{[0, 1]} \phi\cdot\phi\circ T_{(\omega,i)}^n~dm=
	\left\{
	\begin{aligned}
	&2~\text{if}~n<n_c(\omega,i)\\
	&1~\text{otherwise.}
	\end{aligned}
	\right.
	\]
	Therefore, 
		\begin{equation}\label{eq:key}
	\int_{[0, 1]} \psi\cdot\psi\circ T_{(\omega,i)}^n~dm=
         \left\{
	\begin{aligned}
	&1~\text{if}~n<n_c(\omega,i)\\
	&0~\text{otherwise.}
	\end{aligned}
         \right.
	\end{equation}
	Now, observe that 
	\begin{small}
		\[
		\mathds E\left(\sum_{n=0}^{N-1}\psi\circ T_{(\omega,i)}^n \right)^2=\sum_{n=0}^{N-1} \mathds E(\psi^2\circ T_{(\omega,i)}^n)
		+2\sum_{n=0}^{N-1}\sum_{m=n+1}^{N-1}\mathbb E\left(\psi\cdot\psi\circ T_{\sigma^n (\omega,i)}^{m-n}\right),
		\]
	\end{small}
where  $\mathds E$ denotes  the expectation w.r.t. $m$. Since $\mu_{(\omega, i)}=m$ for $(\omega, i) \in \Omega$, we have that 
	 \[\mathds E(\psi^2\circ T_{(\omega,i)}^n)=\int_{[0, 1]} \psi^2 dm=2.\] For the other term, remark that by \eqref{eq:key}
	\begin{equation}
	\sum_{m=n+1}^{N-1}\mathbb E(\psi\cdot\psi\circ T_{\sigma^n(\omega,i)}^{m-n})=\min\left(n_c(\sigma^n (\omega,i)), N-n\right)-1
	\end{equation}
	so that we get
	\[\lim_{N\to \infty} \frac{1}{N}\mathds E\left(\sum_{n=0}^{N-1}\psi\circ T_{(\omega,i)}^n \right)^2=\lim_{N\to \infty} \frac{1}{N}\sum_{n=0}^{N-1}\min\left(n_c(\sigma^n (\omega,i)), N-n\right),\]
	for $\mathbb P$-a.e $(\omega,i)\in\Omega$. On the other hand, since $n_c$ is a measurable, positive and non-integrable function, it follows from Lemma~\ref{Mac} below (applied to $f_m(\omega,i)=\min(n_c(\omega,i),m)$ and $f(\omega,i)=n_c(\omega,i)$) that 
\[
\lim_{N\to \infty} \frac{1}{N}\sum_{n=0}^{N-1}\min\left(n_c(\sigma^n (\omega,i)), N-n\right)=+\infty, \quad \text{for $\mathbb P$-a.e. $(\omega, i)\in \Omega$.}
\]
In particular, we see that the first assertion of Theorem~\ref{TM} does not hold. 
\begin{remark}
Observe that the observable $\psi$ constructed above is deterministic, i.e. it does not depend on $(\omega, i) \in \Omega$. In particular, it satisfies~\eqref{EQ}. Thus, the example we discussed shows that the condition such as~\eqref{eq:condobservable} is needed for Theorem~\ref{TM} to hold. This provides an affirmative answer to the question posed in the first version of~\cite{DDS}.
\end{remark}
\begin{remark}
We note that our example can be further simplified: we can replace $T_1$ with the identity map on $[0, 1]$.
\end{remark}

	\begin{lemma}[Maker's theorem for positive non-integrable functions]\label{Mac}
	
Let $(\mathcal X,\mathcal B,\nu)$ be a probability space and $T:\mathcal X\to\mathcal X$  an ergodic probability preserving transformation. Let $(f_m)_{m\in \N}$ be a sequences of measurable real-valued and nonnegative functions on $\mathcal X$ so that $\lim_{m\to\infty}f_m=f$ exists $\nu$-a.e. and $\int_{\mathcal X} f(x)d\nu(x)=\infty$. Then 
	$$
	\lim_{N\to\infty}\frac 1N\sum_{n=0}^{N-1}f_{N-n}\circ T^n=+\infty, \quad \text{for $\nu$-a.e. $x\in \mathcal X$.}
	$$  
	\end{lemma}
	\begin{proof}
	Firstly, by replacing $f_m$ with $g_m=\inf\{f_n:\,n\geq m\}$, we can assume without any loss of generality that $f_m\leq f_{m+1}$ for all $m\in \N$. Indeed, it suffices to observe that  $g_m$ also converges to $f$ and that 
	$$\frac 1N\sum_{n=0}^{N-1}f_{N-n}\circ T^n\geq \frac 1N\sum_{n=0}^{N-1}g_{N-n}\circ T^n, \quad N\in \N.$$
	
Let us fix some $M>0$ and set $A_M=\{x\in \mathcal X: f(x)\leq M\}$. Set $f_{m}^{(M)}=f_m\cdot \mathbb I_{A_M}$ and $f^{(M)}=f\cdot\mathbb I_{A_M}$, where $\mathbb I_A$ denotes the indicator function  of a set $A$. Then, since $f_m$ is increasing in $m$, we have that 
	$$
	f_m^{(M)}\leq f^{(M)}\leq M.
	$$
	Thus, by applying  the classical Maker's theorem~\cite{Maker}, we obtain that
	$$
	\lim_{N\to\infty}\frac 1N\sum_{n=0}^{N-1}f_{N-n}^{(M)}\circ T^n=\int_{\mathcal X} f^{(M)}d\nu=\int_{A_M}fd\nu, \quad \text{$\nu$-a.e.}
	$$  
	On the other hand, since $f_m\geq f_m^{(M)}$ we have that 
	$$
	\liminf_{N\to \infty} \frac 1N\sum_{n=0}^{N-1}f_{N-n}\circ T^n\geq  \lim_{N\to \infty} \frac 1N\sum_{n=0}^{N-1}f_{N-n}^{(M)}\circ T^n= \int_{A_M}fd\nu.
	$$
	By taking the limit as $M\to\infty$, we conclude that
	$$
	\lim_{N\to\infty}\frac 1N\sum_{n=0}^{N-1}f_{N-n}\circ T^n=\int_{\mathcal X} fd\nu=+\infty, \quad  \text{$\nu$-a.e.}
	$$ 
The proof of the lemma is completed. 
\end{proof}

\section{Appendix B: estimation of $K$}
In this section, we explore possible strategies to estimate the tempered random variable $K$ that appears in the statement of Theorem~\ref{TM}. More precisely, under the additional assumption \[\esssup_{\omega\in\Omega}\|L_\omega\|<+\infty,\]
we show that $K\in L^p$ for small enough $p>0$.
\subsection{General strategy}
	We begin by observing that $\tilde \Pi(\omega)$ given by~\eqref{Pi} satisfies 
\[
\| \tilde \Pi(\omega) h\|_{BV} \le \|h\|_{BV}+ \bigg | \int_X h\, d\mu_\omega \bigg |\le (1+C_{var})\| h\|_{BV},
\]
and
\[
\| (\Id-\tilde \Pi(\omega))h\|_{BV}\le C_{var} \|h\|_{BV}, \quad h\in BV.
\]
Hence, we observe that  $\tilde D$ appearing in the statement of Corollary~\ref{Cor} (see~\eqref{est11} and~\eqref{est22}) can be constructed as:
\[
\tilde D(\omega):=(1+C_{var})\sup_{n\ge 0}(\|L_\omega^n\rvert_{BV_\omega^0}\|e^{\lambda' n}).
\]
We now recall that $K$ is given as a sum of $\tilde D$ and a number $C_{var}$ (see~\eqref{K-exp}). Therefore, from now on we will concentrate on a possible strategy to estimate $\tilde D$, also ignoring  the constant factor $1+C_{var}$. We suppose that 
\begin{equation}\label{bound2}
B:=\esssup_{\omega \in \Omega}\|L_\omega \|<+\infty.
\end{equation}

We begin by noting (see the proof of~\cite[Proposition 28]{DDS}) that $\lambda'$ can be chosen as $\lambda'=-\lambda_2-\epsilon$ for any $\epsilon>0$, where $\lambda_2<0$ is the second largest Lyapunov exponent of the cocycle $(L_\omega)_{\omega \in \Omega}$. By Kingman's subadditive ergodic theorem, we have that 
\[
\lambda_2=\inf_{n\in \N} \frac 1 n\int_\Omega \log \|L_\omega^n \rvert_{BV_\omega^0} \| \, d\mathbb P(\omega).
\]
In particular, there exists $n_0\in \N$ such that 
\[
\frac{1}{n_0}\int_\Omega \log \|L_\omega^{n_0} \rvert_{BV_\omega^0} \| \, d\mathbb P(\omega)<\lambda_2+\frac{\epsilon}{2}.
\]
Set
\[
A_n=\{ \omega \in \Omega: \|L_\omega^n\rvert_{BV_\omega^0}\|>e^{(\lambda_2+\epsilon)n}\}=\bigg \{\omega \in \Omega: \frac 1 n \log \|L_\omega^n\rvert_{BV_\omega^0}\|>\lambda_2+\epsilon \bigg \}.
\]
Set 
\[
g(\omega):=\frac{1}{n_0}\log \|L_\omega^{n_0}\rvert_{BV_\omega^0}\|, \quad \omega \in \Omega.
\]
By $S_n g$ we will denote the $n$-th Birkhoff sum of $g$ with respect to $\sigma$. It follows from~\eqref{bound2} and~\cite[Lemma 3.1]{GS} that there exists $M>0$ such that 
\[
\log \| L_\omega^n\rvert_{BV_\omega^0}\| \le S_n g(\omega)+M,
\]
for $\mathbb P$-a.e. $\omega \in \Omega$ and $n\in \N$. Hence, if $\omega\in A_n$
\[
\frac 1 n S_n g(\omega)+\frac M n >\int_\Omega g\,d\mathbb P+\frac{\epsilon}{2},
\]
so that
\[
\frac 1 n S_n g(\omega)-\int_\Omega g\,d\mathbb P>\frac{\epsilon}{2}-\frac{M}{n}>\frac{\epsilon}{4}
\]
for $n$ sufficiently large. In particular, \[A_n\subset \bigg \{\omega \in \Omega: \frac 1 n S_n g(\omega)-\int_\Omega g\,d\mathbb P>\frac{\epsilon}{4} \bigg \}.\] Provided that $g$ satisfies the large deviation property (as stated in~\cite[Theorem A]{DFGTV2}), there exists $\alpha \in (0,1)$ such that $\mathbb P(A_n)\le \alpha^n$ for $n$ sufficiently large. In particular,  $\sum_{n=1}^\infty \mathbb P(A_n)<+\infty$. By the Borel-Cantelli lemma, we can conclude that for $\mathbb P$-a.e. $\omega \in \Omega$, there exists a smallest $N_\omega \in \N$ such that 
\[
\|L_\omega^n\rvert_{BV_\omega^0}\| \le e^{-\lambda' n}, \quad n\ge N_\omega.
\]
Taking into account~\eqref{bound2} we see that  
\[
\tilde D(\omega) \le e^{(\lambda'+\log B)N_\omega}, \quad \text{for $\mathbb P$-a.e.  $\omega \in \Omega$.}
\]
Moreover, observe that 
$
\{N_\omega = k\} \subset A_{k-1}, 
$ and thus $\mathbb P(\{N_\omega=k\})\le \alpha^{k-1}$ for large $k\in \N$. In particular,  for each $q>0$, we have that there exists $C>0$ such that 
\[
\int_\Omega \tilde D(\omega)^q\, d\mathbb P(\omega)\le C\sum_{k=0}^\infty e^{q(\lambda'+\log B)k}\alpha^{k-1}<+\infty,
\]
provided that 
\[
\log \alpha+q(\lambda'+\log B)<0,
\]
which is satisfied whenever $q>0$ is sufficiently small.

\subsection{A concrete example}

Finally, let us demonstrate the above strategy in a concrete example where the idea can be considerably  simplified. In particular, there is no need to compute the exact value of $\lambda_2$.

Let us assume that $X=[0,1]$ is unit interval and that $\mathbb P$ a.s. we have
$T_\omega\in\{T_1,T_2\}$, where $T_1x=2x\text{ mod }1$ and $T_2$ is the identity map on $X$. We assume also that $\mathbb P(\{\omega: T_\omega=T_i\})>0$ for $i=1,2$.
In this case we have that  $\mu_\omega=m$, where $m$ is the Lebesgue measure on $X$. In particular, $\mathcal L_\omega=L_\omega$ for $\omega \in \Omega$.
Since $T_1$ is expanding,  there are constants $C,\lambda_1>0$ such that for every $k\in \N$ and $h\in BV^0$, we have that 
\begin{equation}\label{99}
\|\mathcal L_{1}^kh\|_{BV}\leq Ce^{-\lambda_1 k}\|h\|_{BV},
\end{equation}
where $\mathcal L_1$ denotes the transfer operator associated to $T_1$.
Let $N_n(\omega)$ be the number of $0\leq j<n$ such that $T_{\sigma^j\omega}=T_1$, i.e.  $N_n(\omega)=\sum_{j=0}^{n-1}\mathbb I_A( \sigma^j\omega)$ where $A=\{\omega: T_\omega=T_1\}$ and $\mathbb I_A$ denotes the indicator function of $A$.  By~\eqref{99} and using that the transfer operator of $T_2$ is the identity operator, we have that 
\begin{equation}\label{999}
\|\mathcal L_\omega^n h\|_{BV}\leq Ce^{-\lambda_1 N_n(\omega)}\|h\|_{BV},
\end{equation}
for $\omega \in \Omega$, $h\in BV^0$ and $n\in \N$.
Let $a:=\mathbb P(A)>0$ and set \[N_\omega=\inf \bigg \{N: N_n(\omega)\geq \frac12 an,\,\,\,\forall n\geq N \bigg\}.\] Then, \eqref{999} implies that for $\mathbb P$-a.e. $\omega \in \Omega$ and $n\in \N$,
$$
\|\mathcal L_\omega^n\rvert_{BV^0}\|\leq K(\omega)e^{-\lambda n},
$$
where 
$K(\omega)=Ce^{\lambda_1 N_\omega}$ and $\lambda=a\lambda_1/2$.
Observe that for $k\ge 1$, \[\{N_\omega=k+1\}\subset \left\{\left|\frac{N_{k}(\omega)}{k}-a\right|>\frac 1 2 a\right\}.\] Thus, if the stationary process $(\mathbb I_A\circ\sigma^n)$ satisfies an appropriate large deviations principle (e.g. choosing maps $T_\omega \in \{T_1, T_2\}$ in an i.i.d. fashion), we can conclude that $N_\omega$ is integrable. Hence, $\log K$ is integrable and consequently also tempered. 
This provides an explicit formula for $K(\omega)$, and our scaling condition~\eqref{eq:condobservable} means that $\|\psi_\omega\|_{BV}$ is small when it takes a lot of time for the Birkhoff average to get close enough to its mean.




	\begin{small}
		\section{Acknowledgements}
		We would like to express our gratitude to Alex Blumenthal for useful discussions related to Appendix B and to referees for their constructive criticism.
		\\D.D. was supported in part by  Croatian Science Foundation under the project IP-2019-04-1239 and by the University of Rijeka under the projects uniri-prirod-18-9 and uniri-prprirod-19-16.
		\\J.S. was supported by the European Research Council (ERC) under the European Union's Horizon 2020 research and innovation programme (grant agreement No 787304).
	\end{small}

	\begin{small}
		\section{Data availability}
		Data sharing not applicable to this article as no datasets were generated or analysed during the current study.
	\end{small}

	\bibliographystyle{amsplain}
	\begin{footnotesize}
		
	\end{footnotesize}
\end{document}